\documentclass[trsc,nonblindrev]{informs3}

\OneAndAHalfSpacedXI 

\usepackage{endnotes}
\let\footnote=\endnote
\let\enotesize=\normalsize
\def\notesname{Endnotes}%
\def\makeenmark{\hbox to1.275em{\theenmark.\enskip\hss}}
\def\enoteformat{\rightskip0pt\leftskip0pt\parindent=1.275em
  \leavevmode\llap{\makeenmark}}

\def\R{\ensuremath{\mathbb{R}}}

\newcommand{\set}[1]{\ensuremath{\mathcal{#1}}}

\usepackage{natbib}
 \bibpunct[, ]{(}{)}{,}{a}{}{,}%
 \def\bibfont{\small}%

\TheoremsNumberedBySection  

\EquationsNumberedThrough    
\usepackage{algorithm} 
\usepackage{color}
\usepackage{algpseudocode}
\usepackage{hyperref}
\usepackage{multirow}
\usepackage{booktabs}
\usepackage{subcaption}
\usepackage{cleveref}

\begin{document}

\RUNAUTHOR{Deng and Pantuso}
\RUNTITLE{Zonification and Pricing in Carsharing}

\TITLE{Zonification and Pricing in Carsharing}

\ARTICLEAUTHORS{%
\AUTHOR{Jiali Deng}
\AFF{Department of Mathematical Sciences, University of Copenhagen, Copenhagen, 2100, \EMAIL{jd@math.ku.dk}} 
\AUTHOR{Giovanni Pantuso}
\AFF{Department of Mathematical Sciences, University of Copenhagen, Copenhagen, 2100, \EMAIL{gp@math.ku.dk}}
} 

\ABSTRACT{In this article we address the problem of partitioning a carsharing business area into pricing zones. 
We formalize the problem mathematically and show that the resulting partitioning problem can be formulated as a binary integer programming problem. The partitioning problem is then extended to include pricing decisions, yielding the first joint zonification and pricing problem. 
The resulting mixed integer (possibly nonlinear) programming problem is solved exactly using an ad-hoc integer Benders decomposition for which we define effective problem-specific improvements. 
Extensive tests based on a real-world carsharing system demonstrate that the method outperforms a state-of-the-art commercial solver on problems of size comparable to those encountered in real-world instances. Furthermore, by jointly optimizing prices and pricing zones, we observe a profit increase of $7.01\%$ compared to a zip code-based prior partition of the business area, and of $25.61\%$ compared to a scenario where the business area is not partitioned.}

\maketitle
\section{Introduction}\label{sec:intro}
Carsharing pricing decisions have attracted significant attention in the research literature, see e.g., \citet{boyaci2019investigating, zhang2022optimization,huang2020vehicle, soppert2022differentiated,Jorge2015TripVariations,pantuso2022exact,Muller23}. They have been identified as a promising instrument to resolve fleet imbalances, see e.g., \citet{illgen2019literature}, and improve profits and service rates. 
Among other things, prices are commonly differentiated geographically, see e.g.,  \citet{Jorge2015TripVariations, huang2020vehicle, pantuso2022exact}, that is, dependent on the origin and/or destination of the rental. This typically implies that the business area is partitioned into distinct pricing zones that are independent of pricing decisions \citep{Jorge2015TripVariations, boyaci2019investigating, li2022dynamic} and provided a priori \citep{lu2021performance,huang2020vehicle, pantuso2022exact}. 
The decision of how to optimally divide a business area into pricing zones has not been investigated in detail.

In this paper, we focus on the problem of partitioning a set of carsharing stations into distinct pricing zones. 
We refer to this problem as the \textit{zonification} problem.
The problem, which is motivated by an underlying industrial case, can be briefly described as follows. 
Consider a one-way station-based carsharing system and assume a given fixed set of stations, see e.g., \Cref{fig:pricingzones:a}. The goal of the service provider is to optimally adjust prices and pricing zones periodically during the day, and for small intervals of time (e.g., every hour), in order to adapt to changes in demand patterns. The prices are differentiated by the origin and/or destination of the trip. This entails informing customers about the current prices from their zone to every other zone upon booking. 
Thus, for each time interval, the problem becomes that of partitioning the stations into pricing zones. 
Each zone is a subset of the stations. However, the resulting partition must be such that the zones created form individual ``islands'' or, in other words, they are ``visually disjoint''. This requirement is motivated by the necessity to communicate the pricing mechanism in an easy and intuitive manner via mobile applications. 
The partition illustrated in \Cref{fig:pricingzones:b} would be acceptable as the zones form detached islands. However, the partition in \Cref{fig:pricingzones:c} would not be acceptable as the areas covered by the individual zones overlap. This particular requirement gives rise to a rich set partitioning problem. As we illustrate in \Cref{sec: literature review}, it shares similarities with tessellation and districting problems, though holding distinct characteristics.

\begin{figure}[t]
     \centering
     \captionsetup[subfigure]{font=footnotesize}
     \begin{subfigure}[b]{0.3\textwidth}
         \includegraphics[width=\textwidth]{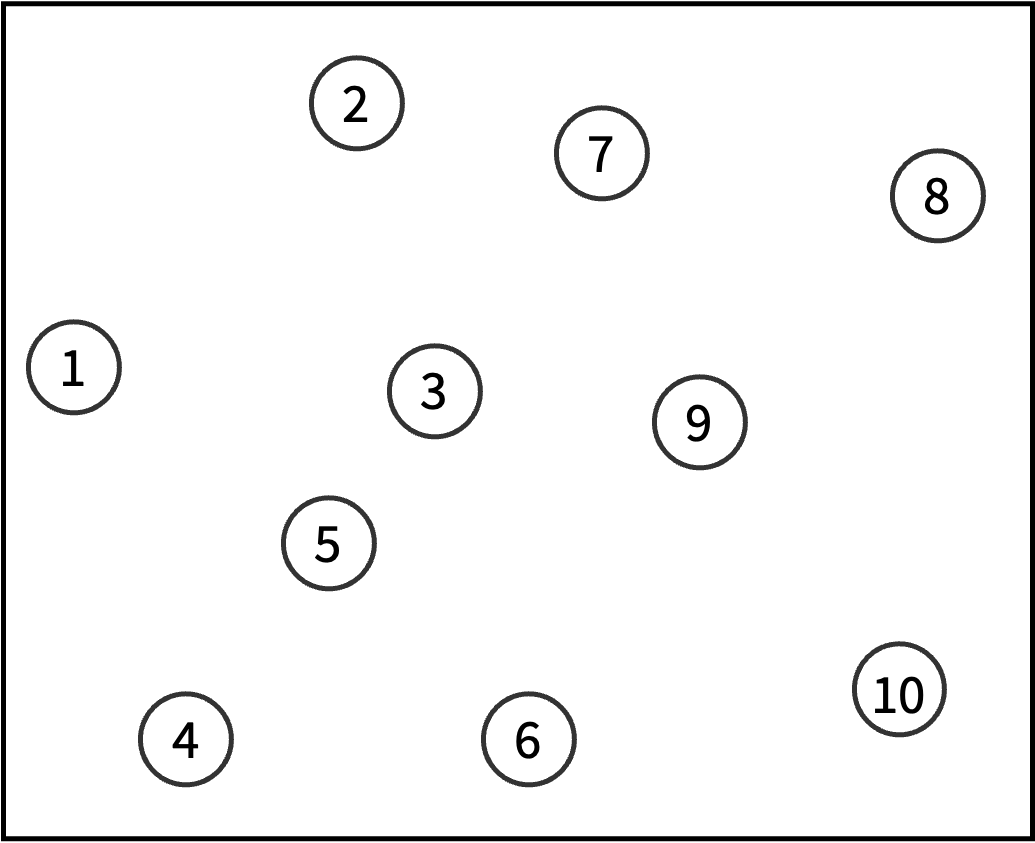}
         \caption{Carsharing stations.}
         \label{fig:pricingzones:a}
     \end{subfigure}
     \hfill
     \begin{subfigure}[b]{0.3\textwidth}
         \includegraphics[width=\textwidth]{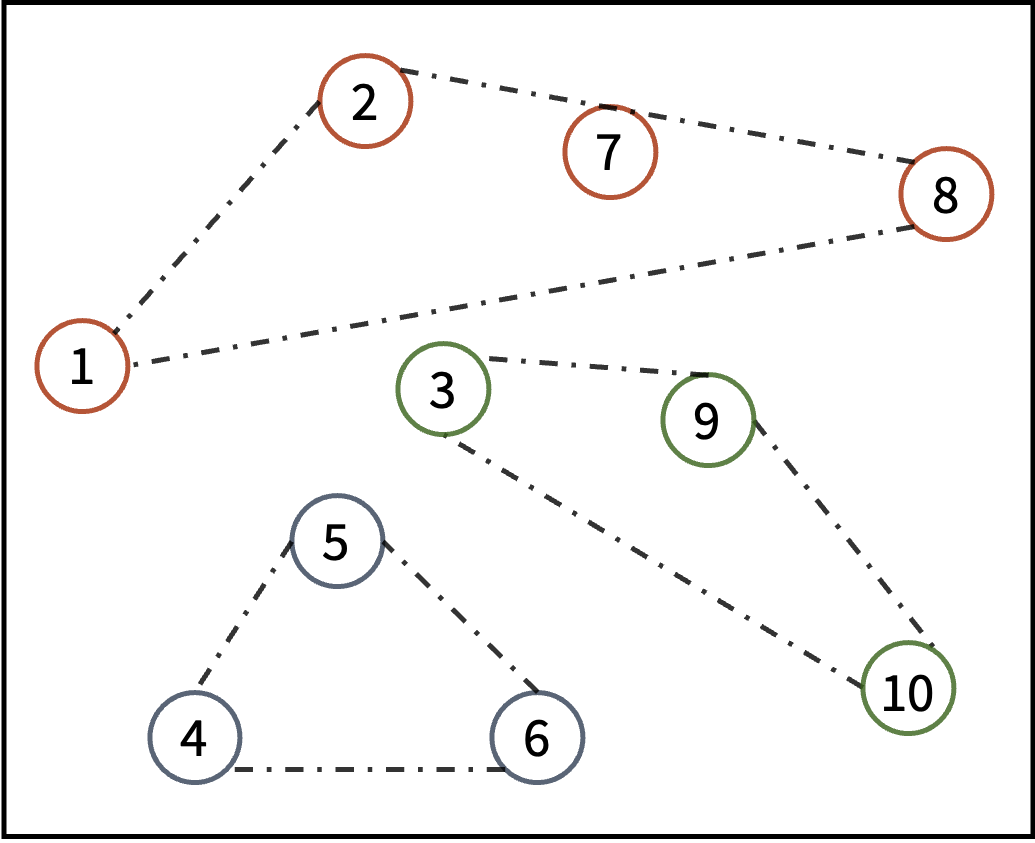}
         \caption{Acceptable partition.}
         \label{fig:pricingzones:b}
     \end{subfigure}
     \hfill
     \begin{subfigure}[b]{0.3\textwidth}
         \includegraphics[width=\textwidth]{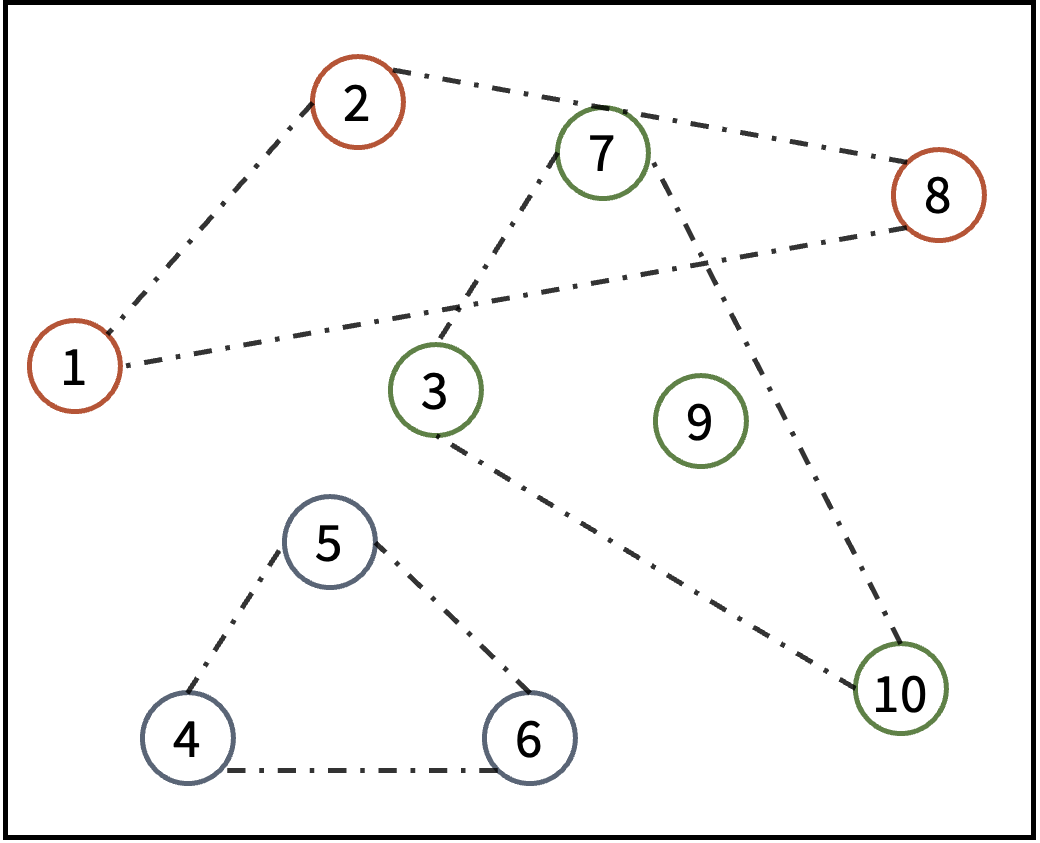}
         \caption{Non-acceptable partition.}
         \label{fig:pricingzones:c}
     \end{subfigure}
        \caption{Carsharing stations, acceptable and non-acceptable partitions.}
        \label{fig:pricingzones}
\end{figure}

In this paper, we give this special partitioning problem a precise mathematical interpretation and express it as a combinatorial optimization problem. We show that this problem can be formulated as a binary integer programming problem.
We then combine the zonification problem with that of setting zone-to-zone prices, hence obtaining a joint zonification and pricing problem. To address the computational complexity of the resulting optimization model, we propose an ad-hoc integer Benders decomposition for which we develop a number of effective improvements. Finally, we test the method on instances derived from a real-life carsharing system.

The remainder of this paper is organized as follows. In \Cref{sec: literature review}, we review the relevant literature and give a more precise account of our contributions with respect to that. In \Cref{sec:zonification}, we formally introduce the zonification problem. In \Cref{sec: pricing problem} we extend the zonification problem to account for pricing decisions yielding a (possibly nonlinear) MIP problem. In \Cref{sec: tailored benders decomposition approach} we present a tailored Benders decomposition, introducing specific efficiency measures. In \Cref{sec: computational experiments} we report on an extensive computational study. Particularly, we shed light both on the efficiency of the proposed decomposition method and on the practical effect of joint zonification and pricing decisions. Finally, we draw conclusions in \Cref{sec: conclusion}.

\section{Literature Review and Contributions} \label{sec: literature review}
In this section we summarize the relevant literature. We start by providing some insights into similar partitioning problems. Following, we summarize the relevant literature on pricing decisions in carsharing. We conclude by summarizing our contributions with respect to the available literature. 

\subsection{Spatial partitioning problems} \label{LR_partition}
The problem of partitioning the carsharing stations into pricing zones shares similarities with well studied problems. These include, in particular, tessellation problems \citep{okabe1997locational,okabe2009spatial} and districting problems \citep{duque2012max,kalcsics2019districting}. 

Tessellation problems can be summarized as follows. Given a finite set of distinct, isolated points in a continuous space, we associate all locations in that space with the closest member of the point set \cite{okabe2009spatial}. The result is a partition of a continuous space. This problem has widespread application and is often referred to with different aliases. The term \textit{Voronoi diagram} appears to be one of the most frequent terms used. Voronoi diagram is a proximity-based partitioning method and provide a convex and compact partition of a continuous space. 
Several variants of Voronoi diagrams exist. These include power Voronoi Diagram and weighted Voronoi Diagram and the associated algorithms. The book \cite{okabe2009spatial} provides a thorough exposition of such variants. 
These concepts have found wide applications also in logistics districting problems, see e.g.,  \cite{galvao2006multiplicatively, novaes2009solving}.

Similar to tessellation problems, we are concerned with partitioning a service region into disjoint areas. 
Nevertheless, there are at least two central differences between our problem and a tessellation problem. 
First, we operate on a discrete space. That is, we need to partition a discrete set of points in $\R^2$ into sub-regions. Second, Voronoi diagrams start from a given set of points (generators) and then assign the remaining points (typically in $\R^2$) to the closest generator. As it will be more clear in \Cref{sec:zonification}, in our problem the set of generators is not given. Choosing the optimal set of generators is thus a central task we need to address. 
Nevertheless, given the similarity between tessellation problems and our problem, we will refer to a solution of our problem as a \textit{discrete tessellation}.

Districting problems are defined as partitioning all basic units in a geographic area into a number of geographic clusters, named districts. Almost all districting approaches require districts to be \textit{contiguous} \citep{kalcsics2019districting}. The definition of contiguity can vary among problems. 
If it is possible to identify, for every basic unit, an explicit neighborhood (e.g., all geographic locations sharing the same zip code or connected by a transportation link), then one can use a graph-based definition of  contiguity. 
Particularly, given a graph that describes neighboring relationships, a district is contiguous if the basic units that form a district induce a connected subgraph. 
Districting problems are often formulated as MIP problems and there exists a wide range of applications. These include, logistics community partition \citep{jarrah2012large, carlsson2013robust}, commercial territory design \citep{salazar2011new, bender2016multi}, political districting \citep{bozkaya2003tabu}, electrical power districting \citep{bergey2003decision, bergey2003simulated}. 
Within these applications, the contiguity criteria is usually satisfied by incorporating an exponential number of constraints \citet{hess1965nonpartisan} based on a neighbourhood graph representation, though polynomial formulations exist \citep{duque2012max}. 
In addition to contiguity, districts are often required to be \textit{compact}. 
Intuitively, a district is said to be geographically compact if it is undistorted (preferably round-shaped) \citep{kalcsics2019districting}. In logistics and transportation applications, compactness acts as a proxy of travel time. There is no uniformly accepted definition of compactness \citep{mehrotra1998optimization}. Mathematically, it has been quantified in different ways \citet{bergey2003decision, bard2009large,hess1965nonpartisan}. The typical way in which compactness is ensured is by minimizing one such measure of compactness. One of the first and most popular measures is the \textit{moment of inertia} which is the squared sum of the distances between the units in the district and center of the district, weighted by the population of the units. This measure is often used in political districting applications, see e.g., \citet{hess1965nonpartisan}. Another example is provided by \citet{salazar2011new} which minimizes the sum of distances between the units in the same district. 
However, in general, minimizing distance-based measures may fail to deliver a partition with non-overlapping districts \cite{kalcsics2019districting}.

Similar to districting problems we are also concerned with partitioning a geographical set of locations in order to obtain somewhat compact subsets. In our case, the requirement of having visually disjoint districts (zones in our case), see \Cref{sec:intro}, leads to our own specific interpretation of compactness. Contrary to the vast majority of districting problems, we are however prevented from ensuring compactness by minimizing a measure of compactness. In our problem a set of compact (according to the definition we will give in \Cref{sec:zonification}) districts or zones is the elementary building block of a pricing problem which pursues maximization of a measure of business performance. We will show that we are able to induce compact zones by means of binary variables and linear constraints .

\subsection{Carsharing pricing}
The pricing policies considered in the literature can be classified along several dimensions. One such classification is attempted by \cite {pantuso2022exact} who divides pricing policies as either \textit{individual} or \textit{collective}. Individual pricing policies are targeted to individual users and require an interaction between the service provider and the individual user. In such interaction, the final trip price is agreed upon. Collective pricing policies are, instead, targeted to the entire user base. They have the scope of influencing the cumulative rental demand by differentiating prices geographically or across time. In this paper we focus on collective pricing policies. 

Among the collective pricing policies, several devise spatially differentiated prices. In these methods, the price changes according to the origin and/or destination of the trip. These methods are, in general, aimed to hedge against potential geographical mismatch between supply and demand.
\citet{waserhole2016pricing} consider two origin-destination based pricing policies, namely static and dynamic. The former disregards changes occurring in the carsharing system and set prices based on average values. The latter adapts to the arrival of new information. Both pricing policies set prices according to individual pairs of carsharing stations.
\citet{ren2019novel} assume a fleet of shared electric vehicles and define prices that vary with time and with the origin and destination of the trip. Particularly, prices are set between pairs of carsharing stations. Their pricing mechanism also takes into account the profit obtained from the vehicle-to-grid interaction. 
\citet{xu2018electric} study fleet sizing and trip pricing decisions for a station-based one-way carsharing system. They assume demand is elastic with respect to price and define prices that can vary according to the origin and destination  station as well as the time and the specific user group.
\citet{angelopoulos2018incentivized} assume a pricing mechanism where price incentives are offered to users traveling from stations with a surplus of vehicle to stations with a shortage of vehicles. 
\cite{zhang2022optimization} consider a station-based carsharing service and assume the price depends on the origin and destination station. User may however receive a price incentive to pick up the car from a nearby station. A common feature in \cite{waserhole2016pricing,ren2019novel,xu2018electric,angelopoulos2018incentivized,zhang2022optimization} is that the prices are set according to specific origin and destination pairs and the business area is not divided into zones.

\citet{huang2020vehicle} consider a station-based carsharing service and explore a pricing mechanism based both on the origin and destination of the rental. Additionally, their investigated pricing mechanism incorporates penalties for leaving and arriving at specific stations. The two penalties are independent on each other. Unlike in the previous studies, the authors assume that the business area is partitioned into zones, and prices depend on the origin and destination zone (or both). 
Pricing zones are also assumed in the work of \citet{hansen2018pricing,pantuso2020formulations}, later extended by \cite{pantuso2022exact}. They propose a pricing mechanism made of a per-minute fee and a drop-off fee. The former is independent while the latter is dependent on the origin and destination zone of the trip. 
 \citet{soppert2022differentiated} differentiate prices for a free-floating carsharing service according to the origin and destination zone of the trip.
\citet{Muller23} consider a customer-centric pricing mechanims where prices can vary according to the origin and time of the rental. They compare the pricing mechanisms to a number of benchmark location-based pricing mechanisms. Among these we find pricing mechanisms where the business area is partitioned into zones of $1km\times 1km$ and the way the price is set for a given zone gives rise to different pricing mechanisms.
\citet{li2022dynamic} consider a one-way station-based service. They assume prices depend on the origin and destination station of the rental. However, to reduce the dimension of the problem, they cluster carsharing stations into zones and assume that the same price applies to any station within the same zone. Zones are determined prior to solving the pricing problem. 
\citet{boyaci2019investigating} consider a one-way, station-based, carsharing service. The pricing mechanism entails offering incentives to users for flexibility in choosing the origin and destination station.
\citet{lu2021performance} consider a one-way station-based carsharing service. They study the problem of determining rental prices in a context of competition with private cars. Particularly, they determine the prices between any pair of zones for each time. Zones are input to the problem. \citet{Jorge2015TripVariations} consider a one-way station-based system. They start by determining the zones of the city by clustering stations based on the similarity of their ideal demand and supply. This is, in turn, computed assuming relocation activities have cost zero. Following, they determine the prices which can vary with the origin and destination zone of the rental as well as with time.

A common characteristic to these studies is that they all assume that the business area is partitioned into pricing zones. Such pricing zones are either given \citep{huang2020vehicle, hansen2018pricing,pantuso2020formulations, lu2021performance, pantuso2022exact, soppert2022differentiated,Muller23} or determined prior to addressing, and independently of, pricing decisions \citep{boyaci2019investigating, li2022dynamic}. We extend this work by jointly addressing zonification and pricing decisions.

\subsection{Contributions}
The contribution of this paper can be stated as follows.
\begin{itemize}
    \item We formalize mathematically the problem of partitioning a set of carsharing stations into non-overlapping pricing zones. We begin with a rather abstract definition of the partitioning problem and prove that a corresponding partition ensures the non-overlapping property. We then show that the set of partitions satisfying the given definition can be modeled using binary variables and linear constraints.
    \item We then extend the zonification problem to include pricing decisions. This extends the available carsharing literature by providing the first joint zonification and pricing problem. The resulting problem is formulated as a (possibly nonlinear) MIP problem. Linearity of the problem depends on the specification of the measure of performance of zonification and pricing decisions.
    \item We provide a tailored integer Benders decomposition to find exact solutions to the problem even when the chosen performance measure yields a nonlinear problem. For this method we incorporate a number of problem-specific improvements.
    \item We provide empirical evidence on both the performance of the solution method and the practical effect of joint zonification and pricing decisions. This is achieved by solving instances generated from a real-world carsharing service based in Copenhagen. 
\end{itemize}

\section{Zonification Problem}\label{sec:zonification}

We start by introducing the problem in an abstract way. 
Later we show how this problem relates to that of defining pricing zones as discussed in \Cref{sec:intro}.

Given a discrete metric space $(\set{I},d)$ we are concerned with the problem of finding a special partition of $\set{I}$ whose characteristics are described in \Cref{def:partition}.
\begin{definition}[Discrete tessellation]
\label{def:partition}
Let $\set{G}\subseteq \set{I}$. A collection $\set{V}(\set{G})\subset 2^{\set{I}}$ of subsets of $\set{I}$ is called the \textit{discrete tessellation of $\set{I}$ induced by $\set{G}$} iff the following properties hold:
\begin{enumerate}
    \item (Disjunction) For every two sets $\set{V},\set{U}\in \set{V}(\set{G})$ we have $\set{V}\cap \set{U}=\emptyset$ 
    \item (Cover) $\bigcup \set{V}(\set{G})=\set{I}$
    \item (One generator) For every $\set{V}\in \set{V}(\set{G})$ we have $\lvert \set{V}\cap \set{G}\rvert= 1$ 
    \item (Closest to generator) For every set $\set{V}\in \set{V}(\set{G})$ let $c\in\set{V}$ be the element such that $\{c\}=\set{V}\cap \set{G}$. Then, for every element $v\in \set{V}$ of the set we have that $d(v,c)\leq d(v,k)$ for all $k\in \set{G}$.
\end{enumerate}

\end{definition}

Properties 1 and 2 define a partition of $\set{I}$ (i.e., a disjoint cover). Property 3 ensures that each set in the partition contains exactly one element of the set $\set{G}$. 
The elements of $\set{G}$ are thus understood as the \textit{generators} of the tessellation. 
Finally, property 4 characterizes the partition as a tessellation. 
It states that each point in $\set{I}$ is assigned to the subset that contains its closest generator from $\set{G}$ in the sense of the metric $d$. Thus, each set of the partition $\set{V}(\set{G})$ contains the points of $\set{I}$ that are closest to the single element of $\set{G}$ in $\set{V}$ than to any other element of $\set{G}$ in the sense of the metric.

We are particularly concerned with the problem of finding an optimal tessellation according to some measure of performance. That is
\begin{equation}
\label{eq:max_partition}
    \max_{\set{G}\in 2^{\set{I}}}\left\{R \big(\set{V}(\set{G})\big)~|~\text{Properties (1)-(4) hold}\right\}
\end{equation}
where $R$ is a mapping from the set of all partitions of $\set{I}$ to the real numbers. 
Thus, the problem can be seen as that of finding generators $\set{G}\subseteq\set{I}$ whose induced tessellation maximizes $R$. Of particular interest are the problems where we look for a partition made of exactly $S$ subsets. That is, where we restrict the maximization over the set $\set{G}^S:=\{\set{G}\in 2^{\set{I}}|\lvert \set{G} \rvert=S\}$.

The abstract problem of finding an optimal discrete tessellation finds a concrete application in the carsharing zonification problem sketched in \Cref{sec:intro}. There, a partition of the carsharing stations must be such that the subsets are ``visually disjoint'', see \Cref{fig:pricingzones:b}.
The mathematical interpretation we give of this requirement, is that every two subsets of the partition have non-overlapping convex hulls. 
\Cref{prop:disjoint_convex_hulls} shows that, by working with partitions of the type introduced in \Cref{def:partition}, we automatically satisfy this property.

\begin{proposition}[Disjoint convex hulls]\label{prop:disjoint_convex_hulls}
Let $\set{I}$ be a finite subset of $\R^n$. Let $\set{V}(\set{G})\subset 2^{\set{I}}$ be the discrete tessellation of $\set{I}$ induced by $\set{G}\subseteq\set{I}$ as defined in \Cref{def:partition}. 
Then, for all sets $\set{V},\set{U}\in\set{V}(\set{G})$ we have
$$iConv(\set{V})\cap iConv(\set{U})=\emptyset$$
where $iConv(\set{V})$ (resp. $iConv(\set{U})$) is the interior of the convex hull $Conv(\set{V})$ (resp. $Conv(\set{U})$) of the points in the set $\set{V}$ (resp. $\set{U}$).
\end{proposition}

\begin{proof}{Proof of \Cref{prop:disjoint_convex_hulls}}
We prove this result by contradiction. 
Assume $\set{V}(\set{G})$ is a discrete tessellation and that $iConv(\set{V})\cap iConv(\set{U})\neq\emptyset$ for some $\set{V}$ and $\set{U}\in \set{V}(\set{G})$. 
Then, there exists either some $v\in\set{V}$ in $iConv(\set{U})$ or some $u\in\set{U}$ in $iConv(\set{V})$.
Assume, without loss of generality, that $v\in iConv(\set{U})$. Then, we also have that $v\in Conv(\set{U})$ since $iConv(\set{U})\subset Conv(\set{U})$.
Furthermore, since $iConv(\set{U})$ is open, there exists $r>0$ such that $\set{B}(v,r)=\{u'\in \R^n\vert d(u',v)<r\}\subseteq iConv(\set{U})$.
Consequently, there exists some $u'\in \set{B}(v,r)$ (which may or may not be in $\set{U}$) such that $$d(v,g^{\set{U}})\leq d(u',g^{\set{U}})$$
where $g^{\set{U}}\in\set{G}$ is the generator of $\set{U}$.

Since $iConv(\set{U})\subset Conv(\set{U})$ then we also have that $u'\in Conv(\set{U})$. 
Thus, there exist $\alpha_i\geq 0$ for $i=1,\ldots,N$, with $\sum_{i=1}^N\alpha_i=1$, such that $u'=\sum_{i=1}^N\alpha_iu_i$ where $u_i\in\set{U}$ for $i=1,\ldots,N$.
Then we have
$$d(v,g^{\set{U}})\leq d(\sum_{i=1}^N\alpha_iu_i,g^{\set{U}})$$
which is only possible if there exists some $\alpha_i> 0$ and $u_i$ such that 
$$d(v,g^{\set{U}})\leq d(u_i,g^{\set{U}})$$
However, since $d(u_i,g^{\set{U}})\leq d(u_i,g)$ for all $g\in\set{G}$, it follows that $d(v,g^{\set{U}})\leq d(v,g)$ for all $g\in\set{G}$. This violates property 4 in \Cref{def:partition}, and contradicts that $\set{V}(\set{G})$ is a discrete tessellation. Hence, we must have that $iConv(\set{V})\cap iConv(\set{U})=\emptyset$.  \Halmos
\end{proof}

\Cref{prop:disjoint_convex_hulls} clarifies that, by solving problems of type \eqref{eq:max_partition}, one automatically ensures that the partition obtained generates convex hulls with disjoint interiors as required. The convex hulls may however intersect on the boundary. 

We proceed by showing that a discrete tessellation can be enforced by integer linear programming constraints. 
We introduce binary variables $a \in \{0,1\}^{|\mathcal{I}| \times |\mathcal{I}|}$. Variable $a_{ii}$ takes value $1$ if $i\in\set{I}$ is designated as a generator, while variable $a_{ij}$ takes value $1$ if element $j$ is a generator and element $i$ belongs to the same subset as $i$. The set of all feasible discrete tessellations made of exactly $S$ subsets can be expressed using $\mathcal{O}(\vert\set{I}\vert^3)$ linear constraints as follows:
$$\set{T} :=
\left\{a\in \{0,1\}^{|\mathcal{I}| \times |\mathcal{I}|}\left\lvert\begin{matrix}\sum_{j\in\mathcal{I}}a_{ii} = S \\
\sum_{j\in\mathcal{I}}a_{ij} =1 &\forall i\in\mathcal{I}\\
a_{ij}\leq a_{jj}, &\forall i,j\in\mathcal{I}\\
d(i,j_1)a_{i,j_1} \leq d(i,j_2)a_{j_2,j_2}+ d(i,j_1) (1-a_{j_2,j_2}) &\forall i,j_1,j_2\in\mathcal{I}
\end{matrix}\right.\right\}$$
Then, the set of all feasible discrete tessellations is obtained as follows. For every $\bar{a}\in\set{T}$ we have 
$\set{G}(\bar{a})=\{i\in\set{I}:\bar{a}_{ii}=1\}$ and 
$$\set{V}(\set{G}(\bar{a}))=\big\{\set{V}_i=\{i\}\cup\{j\in\set{I}\vert \bar{a}_{ij}=1\} 
~\forall i\in\set{G}(\bar{a}) \big\}$$
Thus, we can rewrite problem \eqref{eq:max_partition} as the following integer program
\begin{equation}\label{eq:max_partition_ref}
    \max_{a\in\set{T}}R(a)
\end{equation}
where (with a slight abuse of notation) $R:\set{T}\to \R$ is defined as $R(a):=R(\set{V}(\set{G}(a)))$.


\section{Pricing Problem} \label{sec: pricing problem}
In this section, discrete tessellations of the type introduced in \Cref{sec:zonification} are used to generate pricing zones for a carsharing service. The pricing problem can be thus summarized as that of finding a discrete tessellation of the carsharing stations and assigning prices between any pair of subsets (henceforth zones). The measure of performance $R(\set{V}(\set{G}))$ is now understood as a measure of business performance (e.g., profits) generated by the rentals occurred as a consequence of the defined zones and prices. 


We assume that the service provider may choose the price to apply between any pair of zones from a discrete set. This is consistent, for example, with the mechanism proposed by \citet{soppert2022differentiated,Muller23} where the per-minute price is selected from a finite set of prices, or by \citet{pantuso2022exact} where a drop-off fee is selected from a finite set of prices.
We denote the set of price levels by $\set{L}$. Binary variables $\lambda_{ijl}$ take value $1$ if price $l \in \set{L}$ is applied between the zones generated by $i\in\set{I}$ and $j\in\set{I}$ (if $i$ and $j$ are chosen as generators), $0$ otherwise. 
Following, we define decision variables $\alpha_{ijl}$ to take value $1$ if price level $l$ is adopted between stations $i\in\set{I}$ and $j\in\set{I}$, $0$ otherwise. 
Let $\alpha:=(\alpha_{ijl})_{i,j \in \mathcal{I}, l \in \mathcal{L}}$ and $\lambda:=(\lambda_{ijl})_{i,j \in \mathcal{I}, l \in \mathcal{L}}$.
The pricing problem can thus be expressed as follows:
\begin{subequations}\label{eq:pricing_problem}
\begin{align}
 \max~& Q(a,\lambda,\alpha)\\
  \text{s.t. }~ & \sum_{l \in \mathcal{L}} \lambda_{ijl} \geq a_{ii} + a_{jj} - 1, &\forall i, j \in \mathcal{I} \label{zpc1}\\
& \sum_{l \in \mathcal{L}} \lambda_{ijl} \leq a_{ii}, &\forall  i, j \in \mathcal{I} \label{zpc2}\\
& \sum_{l \in \mathcal{L}} \lambda_{ijl} \leq a_{jj}, &\forall  i, j \in \mathcal{I} \label{zpc3}\\
& a_{i_1,j_1} + a_{i_2,j_2} + \lambda_{j_1,j_2,l} \leq \alpha_{i_1,i_2,l} + 2, &\forall i_1, i_2, j_1, j_2 \in \mathcal{I}, \forall l \in \mathcal{L} \label{spc1}\\
& \sum_{l \in \mathcal{L}} \alpha_{ijl} = 1, &\forall i,j \in \mathcal{I} \label{spc2}\\
&a\in\set{T}& \label{eq:pricing_problem:tass}\\
&\lambda,\alpha\in \{0,1\}^{\vert\mathcal{I}\vert\times\vert\mathcal{I}\vert\times \vert\mathcal{L}\vert}\label{eq:pricing_problem:range}
\end{align}
\end{subequations}
The function $Q(a,\lambda,\alpha)$ represents the performance (e.g., profit or service rate) obtained by the rentals occurred as a consequence of the prices set between each pair of stations. This function is general and can adapt to the specific configuration of the carsharing service. A possible specification will be provided in \Cref{sec:customer_choices}.
Constraints \eqref{zpc1} state that if both $i$ and $j$ are designated as generators of a zone, then a price level must be assigned between the zones they generate. Constraints \eqref{zpc2} and \eqref{zpc3} ensure that at most one price level is chosen between any pair of zones, and no price level is chosen if either $i$ or $j$ are not designated as generators. Constraints \eqref{spc1} ensure that, if station $i_1$ is assigned to zone $j_1$ and station $i_2$ is assigned to zone $j_2$, then the price level between zones $j_1$ and $j_2$ applies to stations $i_1$ and $i_2$. Constraints \eqref{spc2} ensure that exactly one price level is applied between each pair of stations. Finally, constraints \eqref{eq:pricing_problem:tass} ensure that the $a$ variables define a discrete tessellation. 
Observe that problem \eqref{eq:pricing_problem} is, in general, a nonlinear MIP problem.

\section{Tailored Integer Benders Decomposition} \label{sec: tailored benders decomposition approach}
We propose a tailored integer Benders decomposition to obtain exact solutions to problem \eqref{eq:pricing_problem}. 
The method is based on the following assumption.
\begin{itemize}
    \item[\textbf{A1}] $Q(a,\lambda,\alpha)$ can be computed for all $a,\lambda$ and $\alpha$. 
\end{itemize}
Observe that we do not make any restriction regarding the functional form of $Q(a,\lambda,\alpha)$. The method is thus general in the specification of the performance measure. In our experiments in \Cref{sec: computational experiments} we compute $Q(a,\lambda,\alpha)$ by solving a MILP problem.

We start by reformulating problem \eqref{eq:pricing_problem} as follows:
\begin{align}
        \max_{a \in \set{T}, \lambda, \alpha \in \{0,1\}^{|\mathcal{I}| \times |\mathcal{I}| \times |\mathcal{L}|},\phi}\left\{\phi \vert \phi \leq Q(a,\lambda,\alpha), \eqref{zpc1} - \eqref{spc2}\right\} \label{eq:mp_formulation}
\end{align}
This, in turn, allows us to relax constraints $\phi \leq Q(a,\lambda,\alpha)$ and work with the following \textit{relaxed master problem} (RMP):
\begin{align}\label{eq:RMP}
        \max_{a \in \set{T}, \lambda, \alpha \in \{0,1\}^{|\mathcal{I}| \times |\mathcal{I}| \times |\mathcal{L}|},\phi}\left\{\phi \vert \eqref{zpc1} - \eqref{spc2}\right\} \tag{RMP}
\end{align}

In \eqref{eq:RMP}, $\phi$ overestimates the value of $Q(a, \lambda,\alpha)$. The overestimation is corrected by means of the addition of optimality cuts. We devise integer optimality cuts of the type introduced in \citep{laporte1993integer}. 
Let $U$ be a constant such that 
$$\infty>U\geq \max_{a \in \set{T}, \lambda, \alpha \in \{0,1\}^{|\mathcal{I}| \times |\mathcal{I}| \times |\mathcal{L}|}}Q(a, \lambda,\alpha)$$
Given the $k$-th feasible solution $(a^k, \lambda^k,\alpha^k)$, let $\mathcal{A}_k^+ \subseteq \mathcal{I}\times \mathcal{I}$ and $\mathcal{A}_k^- \subseteq  \mathcal{I}\times\mathcal{I}$ be the sets containing the $(i,j)$ pairs such that $a_{ij}^k=1$ and $a_{ij}^k=0$, respectively. 
Similarly, let $\Lambda^+_k \subseteq \mathcal{I} \times \mathcal{I} \times \mathcal{L}$ and $\Lambda^-_k \subseteq \mathcal{I} \times \mathcal{I} \times \mathcal{L}$ be the sets of tuples $(i,j,l)$ for which $\lambda_{ijl}^k=1$ and $\lambda_{ijl}^k=0$, respectively. Finally, let $\Delta^+_k \subseteq \mathcal{I} \times \mathcal{I} \times \mathcal{L}$ and $\Delta^-_k \subseteq \mathcal{I} \times \mathcal{I} \times \mathcal{L}$ be the sets of tuples $(i,j,l)$ for which $\alpha_{ijl}^k=1$ and $\alpha_{ijl}^k=0$, respectively. \Cref{prop_Optcut} defines the optimality cuts.

\begin{proposition}\label{prop_Optcut}
Let $(a^k, \lambda^k,\alpha^k)$ be the $k$-th feasible solution to constraints \eqref{zpc1} - \eqref{eq:pricing_problem:range}. The set of cuts
\begin{align}
\begin{split}
    \phi \leq \big(Q(a^k, \lambda^k,\alpha^k) - U\big)\bigg(\sum_{(i,j) \in \mathcal{A}_k^+}a_{ij} - \sum_{(i,j) \in \mathcal{A}_k^-}a_{ij} + \sum_{(i,j,l) \in \Lambda_k^+} \lambda_{ijl} - \sum_{(i,j,l) \in \Lambda_k^-} \lambda_{ijl} \\+ \sum_{(i,j,l) \in \Delta_k^+} \alpha_{ijl} - \sum_{(i,j,l) \in \Delta_k^-} \alpha_{ijl}\bigg)
    -\big(Q(a^k, \lambda^k,\alpha^k) - U\big)\big(|\mathcal{A}^+_k| + |\Lambda^+_k| + |\Delta^+_k|- 1\big) + U
\end{split} \label{form_Optcut}
\end{align} 
defined for all feasible $(a^k, \lambda^k,\alpha^k)$ solutions to \eqref{zpc1} - \eqref{eq:pricing_problem:range} makes \eqref{eq:RMP} a reformulation of \eqref{eq:mp_formulation}.
\end{proposition} 
\begin{proof}{Proof of Proposition \ref{prop_Optcut}.}
Observe that, since $a$, $\lambda$ and $\alpha$ are binary, the number of feasible solutions to \eqref{zpc1} - \eqref{eq:pricing_problem:range} is finite, and so is the number of cuts \eqref{form_Optcut}. 
Then, it is easy to see that when $(a, \lambda,\alpha)$ is the $k$-th feasible solution the quantity 
$$\bigg(\sum_{(i,j) \in \mathcal{A}_k^+}a_{ij} - \sum_{(i,j) \in \mathcal{A}_k^-}a_{ij} + \sum_{(i,j,l) \in \Lambda_k^+} \lambda_{ijl} - \sum_{(i,j,l) \in \Lambda_k^-} \lambda_{ijl} + \sum_{(i,j,l) \in \Delta_k^+} \alpha_{ijl} - \sum_{(i,j,l) \in \Delta_k^-} \alpha_{ijl}\bigg)$$
reduces to 
$$|\mathcal{A}^+_k| + |\Lambda^+_k| + |\Delta^+_k|$$
and \eqref{form_Optcut} reduces to 
\begin{align*}
    \phi \leq Q(a^k, \lambda^k,\alpha^k)
\end{align*}
Otherwise it reduces to a generally valid upper bound larger than $U$. 
\Halmos
\end{proof}
Thus, optimality cuts \eqref{form_Optcut} are violated by (and thus cut off) solutions $(a^k,\lambda^k,\alpha^k,\phi^k)$ for which $\phi^k> Q(a^k,\lambda^k,\alpha^k)$. They are however redundant, and thus safe, for solutions other than $(a^k,\lambda^k,\alpha^k,\phi)$ for any $\phi$. Convergence of the algorithm is then granted by the existence of only finitely many such cuts.

\subsection{Improvements} \label{sec: better formulations}
In what follows we discuss a number of efficiency measures aimed to improve the integer Benders decomposition method developed. 

\subsubsection{Reformulation of $\set{T}$}\label{sec:reformulation}
The $\mathcal{O}(\vert\mathcal{I}\vert^3)$ formulation of $\set{T}$ can be improved by rewriting some constraints and identifying redundant ones. In particular, consider constraints 
$$d(i,j_1)a_{i,j_1} \leq d(i,j_2)a_{j_2,j_2}+ d(i,j_1) (1-a_{j_2,j_2}) \quad \forall i,j_1,j_2\in\set{I}$$
We can scale down both sides by $d(i,j_1)$. That is:
 \begin{align}
    & d(i,j_1)a_{i,j_1} \leq d(i,j_2)a_{j_2,j_2}+ d(i,j_1) (1-a_{j_2,j_2}) &\forall i,j_1,j_2\in\mathcal{I} \nonumber\\
    & \iff a_{i,j_1} \leq \bigg( \frac{d(i,j_2)}{d(i,j_1)} -1 \bigg)a_{j_2,j_2} + 1 &\forall i,j_1,j_2\in\mathcal{I} \label{eq:transformed_constr}
\end{align}   
Observe that, when $d(i,j_1) \leq d(i,j_2)$, the right-hand side of constraints \eqref{eq:transformed_constr} is always larger than $1$, making constraints \eqref{eq:transformed_constr} redundant.
When $d(i,j_1) > d(i,j_2)$, the coefficient of $a_{j_2,j_2}$ can be rounded down to $-1$. Therefore, we obtain the following reformulation of $\set{T}$:
$$\set{T} :=
\left\{a\in \{0,1\}^{|\mathcal{I}| \times |\mathcal{I}|}\left\lvert\begin{matrix}\sum_{j\in\mathcal{I}}a_{ii} = S \\
\sum_{j\in\mathcal{I}}a_{ij} =1 &\forall i\in\mathcal{I}\\
a_{ij}\leq a_{jj}, &\forall i,j\in\mathcal{I}\\
a_{i,j_1} \leq 1 - a_{j_2,j_2}, &\forall i,j_1,j_2\in\mathcal{I}: d(i,j_1) > d(i,j_2)
\end{matrix}\right.\right\}$$

\subsubsection{Valid inequalities}\label{sec: VIs}
Valid inequalities can be obtained by examining the distance between each pair of stations. Particularly, for each $i \in \set{I}$, let $j^{(n)}_i\in\set{I}$ with $n=1,\ldots,\vert\set{I}\vert-1$ be the index of the $n$-th order statistic of the distance from $i$, that is $d(i,j^{(1)}_i)\leq d(i,j^{(2)}_i)\leq \cdots\leq d(i,j^{(\vert\set{I}\vert-1)}_i)$. Define, for all $i\in\set{I}$, the set 
$$\set{J}_i:=\left\{j^{(|\set{I}|-S+1)}_i,\ldots,j^{(|\set{I}|-1)}_i\right\}$$
Hence, $\set{J}_{i}$ contains the indices of the $S-1$ stations further from $i$.
It is easy to see that a partition of type \eqref{def:partition} in $S$ zones is only possible if the following valid inequalities are respected
\begin{align}\label{eq:vi:furthestzones}
    &\sum_{j\in \set{J}_i} a_{ji} = 0 &\forall i \in \mathcal{I}
\end{align}

Additional valid inequalities can be obtained using upper-bounds on the performance measure $Q$ for each possible set of generators $\mathcal{G} \in \mathcal{G}^S$. 
Let $\overline{P}_{\mathcal{G}}$ represent the performance upper bound for the tessellation induced by $\set{G}$, i.e., when $a_{ii}=1$ for all $i\in\set{G}$. To compute $\overline{P}_{\mathcal{G}}$ observe that, given $\set{G}$, the allocation of stations to zones (i.e., a specification of variables $a$) can be found in $\mathcal{O}(\vert\set{I}\vert^2)$ operations which assign each station to the closest generator $i\in\set{G}$. 
In \Cref{sec:further_efficiency} we explain how upper bound $\overline{P}_{\set{G}}$ can be computed for our specification of $Q(a,\lambda,\alpha)$.

Given these profit upper bounds we can add the valid inequality as follows.
\begin{align}\label{eq:VI_zone}
    \phi \leq \overline{P}_\mathcal{G} + \left( U - \overline{P}_{\mathcal{G}}\right) (S- \sum_{j \in \mathcal{G}} a_{jj}) ~~~~\forall \mathcal{G} \in \mathcal{G}^S
\end{align}
Observe that, when a given $\set{G}$ of generators is enforced we have $\sum_{j \in \mathcal{G}} a_{jj} = S$. Thus $\phi$ is bounded by the corresponding $\overline{P}_{\mathcal{G}}$. When $\set{G}$ is not enforced we have $\sum_{j \in \mathcal{G}} a_{jj} < S$ and the left-hand side of \eqref{eq:VI_zone} reduces to a value that is greater than $U$. 
Finally, observe that it is possible to generate up to $\binom{\vert\set{I}\vert}{S}$ valid inequalities \eqref{eq:VI_zone}.

\section{Computational Experiments} \label{sec: computational experiments}
We perform experiments on instances based on a real carsharing service in the city of Copenhagen, Denmark. The scope of the experiments is twofold. First, we provide empirical evidence on the performance of the decomposition method described in \Cref{sec: tailored benders decomposition approach}. Second, we illustrate the effect of joint zonification and pricing decisions. 

In the remainder of this section, we begin with \Cref{sec:customer_choices} by explaining how we model performance $Q(a,\lambda,\alpha)$ as the rental profits occurred as a consequence of zonification and pricing decisions. In \Cref{sec:further_efficiency} we provide additional efficiency measures for accelerating the integer Benders decomposition based on the specific model of $Q(a,\lambda,\alpha)$. Following, in \cref{sec: instances} we introduce the instances we solve. Finally, in \Cref{sec: computational performance} we provide empirical evidence on the performance of the solution method and in \Cref{sec: result analysis} we discuss the impact of joint zonification and pricing decisions.

\subsection{Model of rental profits}\label{sec:customer_choices}

Given trip prices between any pair of stations (a decision $\alpha$), customers choose a transport mode between their origin and destination. To model rental activities in response to pricing decisions, and hence profits, we follow the recipes provided by \citet{zheng2023many} and \citet{pantuso2022exact}.

According to \citet{zheng2023many} customers choose a transportation mode by minimizing a generalized transport cost which is derived from travel time and travel fares. Particularly, the \textit{Value of Time} (VOT) is used to convert travel time into an equivalent amount of money. VOT can vary across customers and transport mode, see e.g., \citet{rossetti2023commuter}. In general, it is higher for the time spent on walking and waiting than that for the time spent on a given transport means. Let $\set{M}$ be the set of available transport modes (e.g., public transit, carsharing and taxi). For simplicity, let $m_0\in\set{M}$ be the carsharing mode. Let $\set{K}$ be the set of customers.
Let $\mu_{km}^V$ be the in-vehicle VOT of customer $k$ taking mode $m$ and $\mu_{k}^W$ be the walking-and-waiting VOT of customer $k$, respectively. Correspondingly, the travel time of customer $k$ with mode $m$ is divided into in-vehicle time indicated by $T_{km}^V$ and walking-and-waiting time indicated by $T_{km}^W$. The walking-and-waiting time includes the time to reach the transport mode (e.g., station or carsharing station), to commute and possibly wait for the connection, and to reach the final destination (e.g., walking from the station to the final destination). The generalized transport cost can be expressed as
\begin{equation*}
\label{travel_cost_function}
    c_{km} = p_{km} + \mu_{km}^V T_{km}^V + \mu_{k}^W T_{km}^W, ~~\forall k \in \set{K}, \forall m \in \set{M}
\end{equation*}
Here $p_{km}$ is the transport fare associated with each customer $k\in\set{K}$ by taking mode $m \in \set{M}$. For each $k\in\set{K}$ this is given as 
\begin{align*}
p_{km}= \left\{
\begin{aligned}
&P_{km},  & \forall m \in \mathcal{M} \setminus \{m_0\},\\
&P^{CS} T_{km}^V + \sum_{l\in\set{L}}L_l\alpha_{i(k),j(k),l}, &  m = m_0.
\end{aligned} 
\right.
\end{align*}
where, for all modes other than carsharing $P_{km}$ is a fixed known fare. For carsharing ($m=m_0$), the first term of the fare is the per-minute fee $P^{CS}$ paid for the duration of the ride $T_{km}^V$, while the second term is the drop-off fee applied between the customer's origin and destination station, $i(k)$ and $j(k)$, respectively, where $L_l$ is the fee at level $l\in\set{L}$. 

To model the allocation of shared cars to customers, we use the model provided by  \cite{pantuso2022exact}. This assumes that cars are assigned to customers on a first-come-first-served principle. We begin with identifying the set of potential carsharing customers. These are the customers for which there exists a price level $l\in\set{L}$ which makes their VOT for the carsharing service be the smallest among the transport modes. These customers are henceforth named \textit{requests}. The set of requests is defined as follows
\begin{align*}
   \set{R}=\left\{k\in\set{K}: \exists l \in \set{L}, ~ s.t. ~ c_{km_0}  \leq c_{km}, \forall m \in \set{M} \setminus \{m_0\}\right\}
\end{align*}
For each $r\in\set{R}$ we let $i(r)$ and $j(r)$ be the origin and destination station, respectively, of the customer denoted by $k(r)$, and $l(r)$ the highest acceptable price level. 
Furthermore, we let $\set{L}_r = \{l \in \set{L}: L_l \leq L_{l(r)}\}$ denote the set of all acceptable price levels for request $r$.

Given a set of requests, the allocation of vehicles to requests is done according to a first-come-first-served principle. For each request $r\in\set{R}$, we let $\set{R}_{r} = \{q \in \mathcal{R}: i(q)=i(r), k(q)<k(r)\}$ be the set of requests from the same station which have priority over request $r$. We assume that the indices $k$ of the customers represent the order of arrival at the carsharing station (i.e., customer $k$ arrives before customer $k+1$). This is without loss of generality, as the ordering of the customers can be arbitrary and represent different priority relationships.

Let $\set{V}$ represent the set of available shared vehicles. For all $r \in \set{R}$, $v \in \set{V}$, and $l \in \set{L}$, binary variable $y_{vrl}$ indicates whether request $r$ is served by vehicle $v$ at pricing level $l$. Then, the profit obtained as a result of zonification and pricing decisions is obtained by solving the following MILP:
\begin{subequations}\label{eq:allocation}
\begin{align}
&Q(a,\lambda,\alpha) = \max\sum_{r \in \mathcal{R}} \sum_{v \in \mathcal{V}} \sum_{l \in \mathcal{L}_r} R_{rl}^N y_{vrl}  &\label{sp_obj}\\
&\sum_{v\in \mathcal{V}} \sum_{l \in \mathcal{L}_r} y_{vrl} \leq 1, &\forall r \in \mathcal{R}\label{con_yr}\\
&\sum_{r \in \mathcal{R}} \sum_{l \in \mathcal{L}_r} y_{vrl} \leq 1, &\forall v \in \mathcal{V} \label{con_yv}\\
& \sum_{v \in \mathcal{V}} y_{vrl} \leq \alpha_{i(r),j(r),l}, &\forall r \in \mathcal{R}, l \in \mathcal{L}_r \label{con_yl}\\
&\sum_{l \in \mathcal{L}_{r}}y_{vrl} + \sum_{r_1 \in \mathcal{R}_r} \sum_{l_1 \in \mathcal{L}_{r_1}} y_{v,{r_1},l} \leq G_{v,i({r})}, &\forall r \in \mathcal{R}, v \in \mathcal{V} \label{con_z}\\
& y_{vrl} + \sum_{r_1 \in \mathcal{R}_{r}} \sum_{l_1 \in \mathcal{L}_{r_1}} y_{v,r_1,l_1} + \sum_{v_1 \in \mathcal{V}\setminus\{v\}} y_{{v_1},r,l} \geq \alpha_{i(r),j(r),l} + G_{v,i(r)} - 1, &\forall r \in \mathcal{R}, v \in \mathcal{V}, l \in \mathcal{L}_{r} \label{con_l}\\
&y_{vrl}\in\{0,1\}^{\lvert\mathcal{V}\rvert \times \vert\mathcal{R}\vert \times \vert\mathcal{L}_{r}\vert} &
\end{align}
\end{subequations}
In objective function \eqref{sp_obj}, $R_{rl}^N$ represents the net revenue obtained from serving request $r$ at price level $l$. It is calculated as 
$$R^{N}_{rl} = c_{k(r),m_0} - C_{i(r), j(r)}^U = P^{CS} T_{k(r),m_0}^V + L_l - C_{i(r), j(r)}^U$$
with $C_{i(r), j(r)}^U$ indicating the operating cost born by the carsharing operator when a vehicle is rented between stations $i(r)$ and $j(r)$.
Constraints \eqref{con_yr} ensure that each request is satisfied at most once, while constraints \eqref{con_yv} ensure that each shared vehicle satisfies at most one request. 
Constraints \eqref{con_yl} state that a given request $r$ can be satisfied at a price level $l$ only if the same price level is set between its origin and destination. Constraints \eqref{con_z} ensure that a shared vehicle $v$ can be used to satisfy request $r$ only if it was not occupied by any other customers arriving at the same station earlier. Here, parameter $G_{vi}$ takes value $1$ if vehicle $v$ is at station $i$, 0, otherwise. 
Finally, constraints \eqref{con_l} will force $y_{vrl}$ to take value $1$ if (i) the vehicle $v$ has not been rented by any customer with higher priority (i.e., the second term on the left-hand side is equal to $0$), (ii) the request $r$ has not been satisfied by any other shared vehicles (i.e., the third term on the left-hand side is equal to $0$), (iii) the price level $l$ is offered to all trips between stations $i(r)$ and $j(r)$ (i.e., $\alpha_{i(r),j(r),l} = 1$), and (iv) vehicle $v$ is available at station $i(r)$ (i.e., $G_{v,i(r)} = 1$). 

Observe that problem \eqref{eq:allocation} is always feasible. Its optimal solution can be found in $\mathcal{O}(|\mathcal{R}||\mathcal{V}|)$ operations by the algorithm provided in \Cref{app:alg}. Given this specification of $Q(a, \lambda, \alpha)$, model \eqref{eq:pricing_problem} becomes a MILP. Its extensive formulation is provided in \Cref{app: MILP form}.

\subsection{Further efficiency measures}\label{sec:further_efficiency}
In our implementation of the integer Benders decomposition algorithm, \eqref{eq:RMP} is solved in  a branch-and-bound framework where optimality cuts \eqref{form_Optcut} are identified and added upon reaching integer feasible nodes in the tree. 

Valid inequalities \eqref{eq:vi:furthestzones} and \eqref{eq:VI_zone} are added statically to \eqref{eq:RMP} already at the root node. For valid inequality \eqref{eq:VI_zone} we compute the necessary upper bounds $\overline{P}_{\mathcal{G}}$ as follows. Let $\set{V}(\set{G})$ be the tessellation induced by $\set{G}$, see \Cref{def:partition}. 
Observe that, given our definition of $R^N_{rl}$ in \Cref{sec:customer_choices}, the profit for each request $r$ depends solely on the origin and destination of the request, $i(r)$, $j(r)$, as well as on $l$. That is, it is independent on the customers. Let us refer to $R^N_{rl}$ as $R^N_{ijl}$. Then, for each pair of zones $\mathcal{V}, \mathcal{U} \in \mathcal{V(G)}$ and for each price level $l \in \set{L}$, an upper bound on the profit between such zones at price level $l$ can be computed as follows
\begin{align*}
    \sum_{(i,j)\in \set{V}\times \set{U}}R^N_{ijl}\min \bigg\{\big\vert \{r\in \set{R}: i(r)=i, j(r)=j, l(r)\geq l\}\big\vert, \big\vert \mathcal{V}_i\big\vert \bigg\}
\end{align*}
That is, the upper bound is given by the profit obtained by all requests going from stations in $\set{V}$ to stations in $\set{U}$, provided that there are sufficient vehicles at the origin station (i.e., constrained by the number of vehicles $\vert \set{V}_i \vert$). 
Then, for each pair of zones $(\mathcal{V}, \mathcal{U})$, we identify the price level which gives the highest profit upper bound as 
$$l_{\mathcal{V},\mathcal{U}}^{*}=\argmax_{l\in\set{L}}\left\{\sum_{(i,j)\in \set{V}\times \set{U}}R^N_{ijl}\min \bigg\{\big\vert \{r\in \set{R}: i(r)=i, j(r)=j, l(r)\geq l\}\big\vert, \big \vert \mathcal{V}_i\big\vert \bigg\}\right\}$$
Then $\overline{P}_\mathcal{G}$ is obtained by summing up the profit upper bounds from all pairs of zones as follows.
    \begin{align*}
        \overline{P}_{\mathcal{G}} = \sum_{\mathcal{V}, \mathcal{U} \in \mathcal{V(G)}} \sum_{(i,j)\in \set{V}\times \set{U}}R^N_{ijl_{\mathcal{V},\mathcal{U}}^{*}}\min \bigg\{\left\lvert \{r\in \set{R}: i(r)=i, j(r)=j, l(r)\geq l_{\mathcal{V},\mathcal{U}}^{*}\}\right\rvert, \left\vert \mathcal{V}_i\right\vert \bigg\}
    \end{align*}

In addition, we add classical duality-based Benders decomposition cuts \eqref{Relax_Cut} generated from the LP relaxation of problem \eqref{eq:allocation}.
\begin{align}
    \phi \leq &\sum_{r \in \mathcal{R}} \pi_{rk}^A + \sum_{v \in \mathcal{V}} \pi_{vk}^B \nonumber + \sum_{r \in \mathcal{R}} \sum_{l \in \mathcal{L}_r} \alpha_{i(r),j(r),l} \pi_{rlk}^C \nonumber + \sum_{r \in \mathcal{R}} \sum_{v \in \mathcal{V}} G_{v,i(r)} \pi_{vrk}^D \nonumber\\
       &+\sum_{v \in \mathcal{V}} \sum_{r \in \mathcal{R}} \sum_{l \in \mathcal{L}_r} (\alpha_{i(r),j(r),l} + G_{v,i(r)} -1 ) \pi_{vrlk}^E\label{Relax_Cut}
\end{align}
where, $\pi_{rk}^A$, $\pi_{vk}^B$, $\pi_{rlk}^C$, $\pi_{vrk}^D$, and $\pi_{vrlk}^E$ are the optimal dual solutions at iteration $k$ corresponding to constraints \eqref{con_yr}, \eqref{con_yv}, \eqref{con_yl}, \eqref{con_z}, and \eqref{con_l}, respectively.
Relaxation cuts define a non-trivial upper bound on $Q(a, \lambda,\alpha)$. 
They are generated and added to \eqref{eq:RMP} once an integer feasible solution $(a^k, \lambda^k,\alpha^k,\phi^k )$ to \eqref{eq:RMP} violates the optimality condition $\phi^k \leq Q(a^k, \lambda^k,\alpha^k)$. However, to prevent excessive growth of the size of \eqref{eq:RMP} and avoid solving an excessive number of LPs, we control the frequency at which relaxation cuts are added. 
Let $S_{min}$ denote the least number of zones and $K_{min}$ the least number of customers in our instances. Let $\sigma_{min}$ denote the frequency at which relaxation cuts in the instances with $S_{min}$ and $K_{min}$. This entails that we only add relaxation cuts every $\sigma_{min}$ integer-feasible nodes visited in the branch-and-bound tree.
Let $S_{max}$ and $K_{max}$ the largest number of zones and customers, respectively, in our instances. For a given instance $n$, let $S_n$ and $K_n$ be the number of zones and customers, respectively. The frequency is adjusted according to \eqref{adaptive frequency}.
\begin{align}
\label{adaptive frequency}
    \sigma_n = \sigma_{min} \left(1 + \frac{S_n - S_{min}}{S_{max} - S_{min}}\right) \left(1 + \frac{K_n - K_{min}}{K_{max} - K_{min}}\right) \gamma
\end{align}
Equation \eqref{adaptive frequency} ensures that with each unit increment in problem size, either in terms of the number of zones or customers, the value of $\sigma_n$ is multiplied by the corresponding increase rate. Parameter $\gamma$ controls the rate of growth as the problem size increases. We set $\gamma$ to $0.8$ and $\sigma_{min}$ to $20$.

Finally, since in our instances the fleet is homogeneous, a valid upper bound $U$ for the expression of the optimality cuts \eqref{form_Optcut} can be set as 
\begin{align*}
    U = \sum_{r \in \mathcal{R}} \max\{R_{r, l(r)}^N, 0\}
\end{align*}
It entails that all the requests contributing to a positive revenue are served, and only those, and that these are served at the respective highest price level.

\subsection{Instances} \label{sec: instances}
We build instances based on a real-world carsharing service operating in Copenhagen, Denmark. The instances comprise the twenty stations illustrated in \Cref{fig:cop_css}. 
The transport modes available in the city (set $\set{M}$) comprise public transport (a service offering busses, metro and superficial trains) and taxi, in addition to carsharing. 

\begin{figure}[ht]
    \centering
    \includegraphics[width=0.5\linewidth]{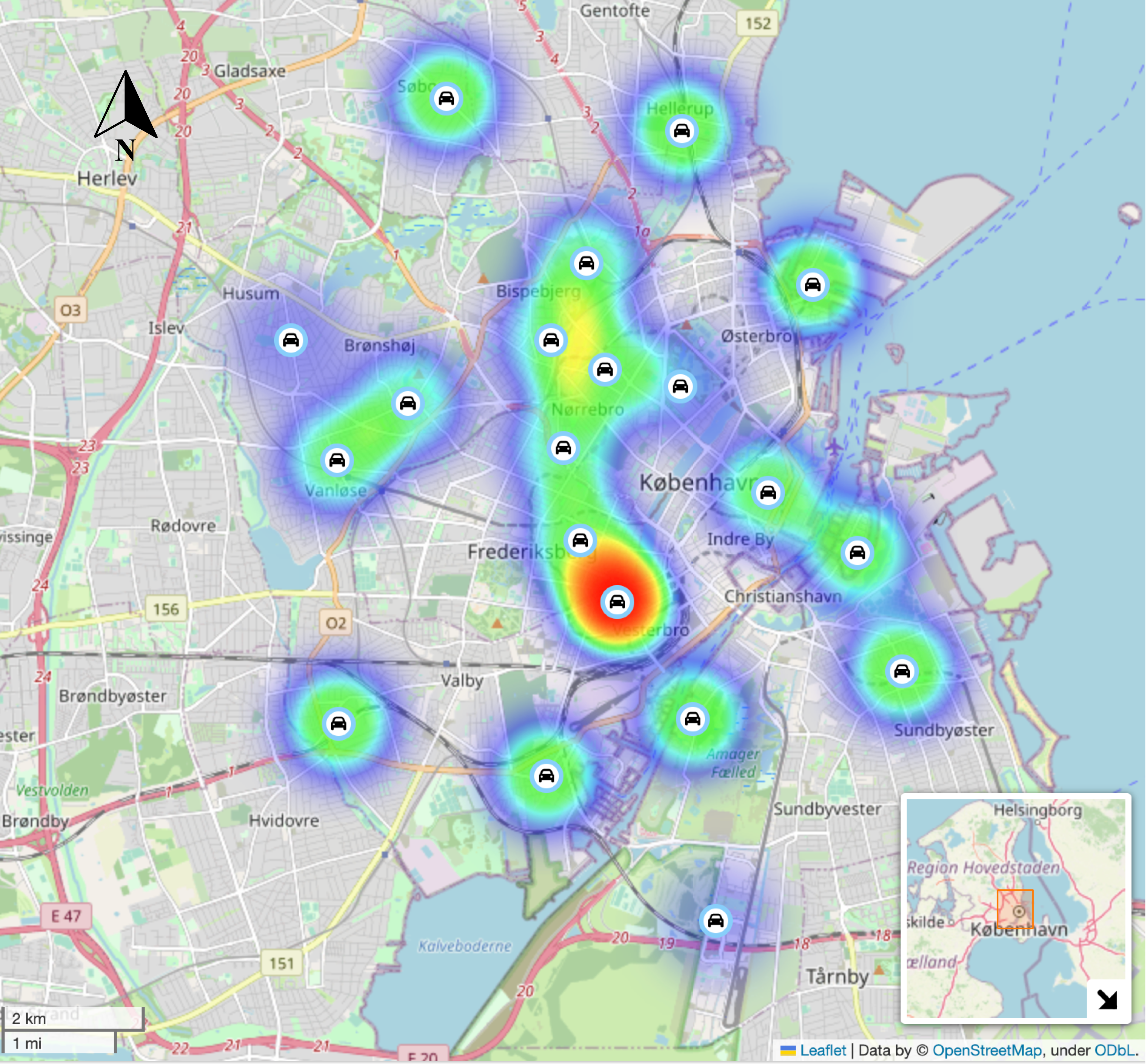}
    \caption{Location of carsharing stations and distribution of POIs.}
    \label{fig:cop_css}
\end{figure}

For the carsharing service we set five drop-off fees, ranging from Euro $-1$ to $3$ with Euro $1$ intervals. The per-minute fee is set at Euro $0.30$ consistently with ongoing practice. The tariffs for taxis and public transit are likewise set according to current tariffs in Copenhagen. These are summarized in \Cref{tab:para}. 

Customers are created based on the points of interest (POIs) located in an area of radius 800 meters centered around each station. POIs are locations such as schools, hospitals and museums and we assume they represent the origins and destinations of the customers. \Cref{fig:cop_css} provides a heatmap of the POIs considered, thus of the origins and destinations of the customers. Particularly, each customer is created by randomly sampling, without replication, an origin and destination pair from the POIs.

For any two POIs, we obtain real-world walking times and travel times and for the different transport modes using Google Maps APIs. 
For the carsharing mode, the waiting time is set to zero. 
For public transit, Google Maps APIs return the overall travel time including both walking-and-waiting time and in-vehicle time. Thus, we assume the walking-and-waiting time is uniformly distributed between $4$ and $15$ minutes and is subtracted from the overall travel time. This represents the time necessary to walk to the a nearby public transit station, to switch between public transit services (e.g., between bus and metro), to reach the next public transit station, and finally to walk to the destination POI. Finally, the waiting time for a taxi is uniformly distributed between $4$ and $8$ minutes while the walking time is zero.

\begin{table}[t]
    \footnotesize
    \caption{Transport prices.}
    \label{tab:para}
    \centering
    \begin{tabular}{l|l}
    \hline
    Parameter & Values\\
    \hline
    Carsharing drop-off fee & -1, 0, 1, 2, 3 Euro for each pricing level\\
    Carsharing per-minute fee & 0.30 Euro\\
    Taxi pick-up fee & 3.89 Euro\\
    Taxi per-minute fee & 2.55 Euro\\
    Public transit ticket fee & 3.22 Euro\\
    \hline
    \end{tabular}
    
\end{table}

VOTs depend on the transport mode. Following \citet{rossetti2023commuter}, we assume the VOTs are distributed according to a log-normal distribution with standard deviation Euro $0.4$ and different mean values. Particularly, the mean values are Euro $2.86$ for the in-shared-vehicle time, Euro $2.94$ for the in-vehicle time of other transport modes, and Euro $4.25$ for the walking-and-waiting time, see \Cref{fig:log_normal}.  Consequently, for each customer, we randomly sample a realization of the VOTs from the distributions above.

\begin{figure}[t]
    \centering
    \includegraphics[width=0.5\linewidth]{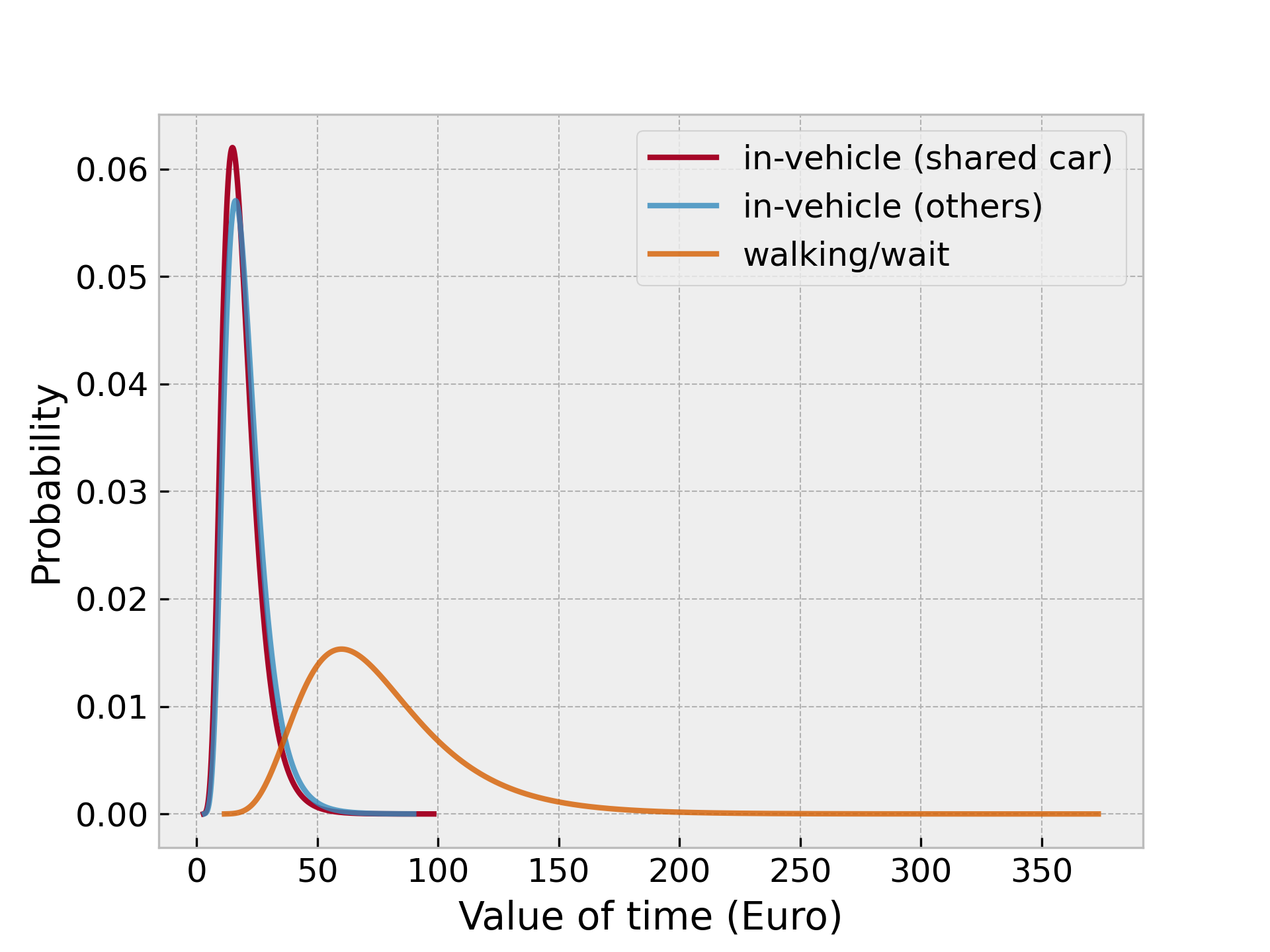}
    \caption{Log-normal distributions of VOT values for different transport modes.}
    \label{fig:log_normal}
\end{figure}

We build two groups of instances. The first group comprises small instances with $10$ stations. The second group comprises larger instances with $20$ stations.
Within each instance group we generate $27$ instances that vary in the number of zones, customers and shared vehicles. In both instance groups, we consider a number of zones $S\in\{ 3, 4, 5\}$. In order to conduct experiments with different ratios between the number of customers and the fleet size, small instances are generated by considering the following combinations of number of customers and number of vehicles, $(\vert\set{K}\vert, \vert\set{V}\vert) \in $ \big\{(100, 20), (100, 30), (100, 40), (200, 40), (200, 50), (200, 60), (300, 60), (300, 80), (300, 100)\big\}. Larger instances are likewise created from the following configurations $(\vert\set{K}\vert, \vert\set{V}\vert) \in$ \big\{(400, 100), (400, 150), (400, 200), (600, 150), (600, 200), (600, 300), (800, 200), (800, 300), (800, 400)\big\}. 

\subsection{Performance of the integer Benders decomposition} \label{sec: computational performance}
In this section we report on the performance of the integer Benders decomposition algorithm (BD in what follows). 
Such performance is compared to that of commercial solver Gurobi 9.5.1 (the solver, in what follows) employed to solve the extensive form of problem \eqref{eq:pricing_problem}. The extensive form of the problem is a MILP and is provided in \Cref{app: MILP form}. 
The BD algorithm is implemented in Python 3.9 using Gurobi's callable libraries. The extensive formulation is likewise solved using Gurobi's Python libraries. In all tests we set a target optimality gap of $0.5\%$, and a time limit of $1800$ seconds. 
All tests were executed on a computer equipped with an Apple M1 Pro processor and $16$ gigabytes of memory. 

\subsubsection{Performance on the small instances}\label{sec:results_small}
We begin by reporting on the results for the small instances (i.e., with $10$ carsharing stations). 
\Cref{tab:comp_small} reports, for each instance, the optimal objective values and the solution times with the solver, BD without the addition of valid inequalities (VIs), and BD with the addition of VIs. When using VIs, we add VIs \eqref{eq:vi:furthestzones} as well as \eqref{eq:VI_zone} for all $\binom{\vert\set{I}\vert}{S}$  possible sets of generators. In both BD with and without VIs we add relaxation cuts as explained in \Cref{sec:further_efficiency}. \Cref{tab:comp_small} does not report optimality gaps as all instances have been solved to an optimality gap below the target tolerance ($0.5\%$). 


\begin{table}[ht]
    \scriptsize
    \centering
    \caption{Objective values and solution times on the small instances.}
    \label{tab:comp_small}
    \begin{tabular}{ccccccccc}
    \toprule
         &&& & Objective Value & & \multicolumn{3}{c}{Time (seconds)}  \\
         \cmidrule{7-9}
         $\vert\set{K}\vert$ & $|\mathcal{V}|$ & $|\mathcal{I}|$ & $S$ & (Euro) & & Solver & BD without VIs & BD with VIs  \\
         \hline
         100 & 20 & 10 & 3 & 52.32 & & 2.70 & 17.87 & 7.69 \\
         100 & 30 & 10 & 3 & 57.30 & & 5.35 & 22.97 & 4.35 \\
         100 & 40 & 10 & 3 & 60.28 & & 6.21 & 11.97 & 3.25\\
         & & & & & \textbf{Average} & \textbf{4.75} & \textbf{17.59} & \textbf{5.10} \\
         200 & 40 & 10 & 3 & 85.56 & & 6.05 & 14.83 & 5.29 \\
         200 & 50 & 10 & 3 & 85.56 & & 6.12 & 33.61 & 8.84 \\
         200 & 60 & 10 & 3 & 90.28 & & 6.19 & 30.37 & 4.31 \\
         & & & & & \textbf{Average} & \textbf{6.12} & \textbf{26.27} & \textbf{6.14} \\
         300 & 60 & 10 & 3 & 133.58 & & 60.21 & 336.77 & 117.89 \\
         300 & 80 & 10 & 3 & 143.98 & & 61.67 & 352.35 & 70.25 \\
         300 & 100 & 10 & 3 & 149.52 & &32.47 & 49.49 & 38.33 \\
         & & & & & \textbf{Average} & \textbf{51.45} & \textbf{246.19} & \textbf{75.49} \\
         \cmidrule{6-9}
         & & & & & \textbf{Average} & \textbf{20.77} & \textbf{96.68} & \textbf{28.91}\\
         \midrule
         100 & 20 & 10 & 4 & 52.32 & & 1.60 & 18.12 & 11.68 \\
         100 & 30 & 10 & 4 & 60.76 & & 2.55 & 5.76 & 4.62 \\
         100 & 40 & 10 & 4 & 63.90 & & 2.13 & 4.20 & 3.18 \\
         & & & & & \textbf{Average} & \textbf{2.10} & \textbf{9.36} & \textbf{6.49} \\
         200 & 40 & 10 & 4 & 90.26 & & 1.86 & 5.65 & 5.45 \\
         200 & 50 & 10 & 4 & 90.26 & & 3.84 & 7.35 & 4.62 \\
         200 & 60 & 10 & 4 & 95.44 & & 5.00 & 4.51 & 3.45 \\
         & & & & & \textbf{Average} & \textbf{3.57} & \textbf{5.84} & \textbf{4.51} \\
         300 & 60 & 10 & 4 & 141.20 & & 66.96 & 321.30 & 191.64\\
         300 & 80 & 10 & 4 & 152.86 & & 35.12 & 291.15 & 93.42 \\
         300 & 100 & 10 & 4 & 158.98 & & 17.45 & 79.29 & 15.97\\
         & & & & & \textbf{Average} & \textbf{39.84} & \textbf{230.58} & \textbf{100.34} \\
         \cmidrule{6-9}
         & & & & & \textbf{Average} & \textbf{15.17} & \textbf{81.93} & \textbf{37.11} \\
         \midrule
         100 & 20 & 10 & 5 & 52.32 & & 1.42 & 10.51 & 15.55 \\
         100 & 30 & 10 & 5 & 62.76 & & 1.45 & 4.75 & 4.20 \\
         100 & 40 & 10 & 5 & 65.90 & & 1.52 & 3.05 & 1.79\\
         & & & & & \textbf{Average} & \textbf{1.46} & \textbf{6.11} & \textbf{7.18} \\
         200 & 40 & 10 & 5 & 90.48 & & 1.39 & 3.88 & 3.75 \\
         200 & 50 & 10 & 5 & 91.70 & & 1.26 & 3.13 & 3.51 \\
         200 & 60 & 10 & 5 & 96.88 & & 1.29 & 1.93 & 2.58 \\
         & & & & & \textbf{Average} & \textbf{1.32} & \textbf{2.98} & \textbf{3.28} \\
         300 & 60 & 10 & 5 & 145.82 & & 18.79 & 442.80 & 356.34 \\
         300 & 80 & 10 & 5 & 158.88 & & 10.80 & 220.63 & 189.26 \\
         300 & 100 & 10 & 5 & 167.74 & & 6.68 & 74.49 & 24.93 \\
         & & & & & \textbf{Average} & \textbf{12.09} & \textbf{245.97} & \textbf{190.18}\\
          \cmidrule{6-9}
         & & & & & \textbf{Average} & \textbf{4.96} & \textbf{85.02} & \textbf{66.88} \\
         \bottomrule
    \end{tabular}
    
\end{table}

We observe that the small instances are solved rather efficiently by the solver. The average computation time decreases as the number of zones increases. When increasing the number of zones we progressively relax the feasible region by including a larger number of partitions (observe that $\binom{\vert\set{I}\vert}{S}<\binom{\vert\set{I}\vert}{S+1}$ always holds in our instances and in general when $\vert\set{I}\vert>2S+1$). At the same time the number of decision variables and constraints is unaffected. The solver outperforms BD on all small instances. 

The performance of BD improves substantially with the addition of VIs. When using VIs the computation time is smaller in $23$ out of the $27$ and only marginally longer in the remaining $4$ instances. The average reduction of the solution time is substantial in different instance sizes and tends to increase as $S$ decreases. Particularly, when using VIs, the computation time decreases by $22.68$\% with $S=5$, $54.71$\% with $S=4$, and $70.10$\% with $S=3$.

Finally, we observe that the optimal objective value increases with $S$. This is due to the fact that the number of feasible solutions increases with $S$. 

\subsubsection{Performance on the large instances}\label{sec:results_large}
We focus now on the instances with $20$ stations. We compare the performance of the solver to that of BD. For BD we use VIs as they were proven effective on the small instances, see \Cref{sec:results_small}. \Cref{comp:large_instance_mean} reports the average optimality gaps aggregated by number of customers and number of zones.

We observe that on these instances BD significantly outperforms the solver. The average optimality gap of BD is $7.47$\% while that of the solver is $29.07$\%. The optimality gap reduction becomes more pronounced when $S$ decreases, thus when the feasible region is smaller. In particular, the optimality gap is decreased by $91.58\%$ with $S=3$, $70.03\%$ with $S=4$, and $52.45\%$ with $S=5$.
The decrease in the optimality gap is consistent across different numbers of customers.
Furthermore, BD successfully closes the optimality gap (within the target tolerance) in 6 out of the 27 instances, while none of these instances is solved to optimality by the solver. Particularly, we observe that BD typically significantly better on the instances with $S=3$, which appear to be particularly challenging for the solver (see also \Cref{tab:comp_small}).
\begin{table}[t]
    \scriptsize
    \caption{Average optimality gaps after $1800$ seconds. Optimality gaps are computed as $\vert$\texttt{best\_bound}-\texttt{objective\_value}$\vert$/$\vert$\texttt{objective\_value}$\vert$. The column ``Reduction'' reports the difference in optimality gap between BD and the solver, as a percentage. ``Instances solved'' reports the number of instances solved to an optimality gap smaller than the target $0.5\%$ tolerance.}
    \label{comp:large_instance_mean}
    \centering
    \begin{tabular}{ccccccccc}
    \toprule
        && & &\multicolumn{3}{c}{Gap (\%)} & \multicolumn{2}{c}{Instances solved}\\
        \cmidrule{5-9}
        $|\mathcal{K}|$ & $|\mathcal{I}|$ & $S$ & & Solver & BD & Reduction (\%) & Solver& BD\\
        \midrule
        400 & 20 & 3 & & 29.28 & 1.37 & 95.32 & 0/3 & 2/3\\
        600 & 20 & 3 & & 38.03 & 3.25 & 91.45 & 0/3 & 2/3\\
        800 & 20 & 3 & & 38.56 & 4.30 & 88.84 & 0/3 & 1/3\\
        \cmidrule{4-9}
        & & & \textbf{Average} & \textbf{35.29} & \textbf{2.97} & \textbf{91.58} & \textbf{0/9} & \textbf{5/9}\\
        \midrule
        400 & 20 & 4 & & 21.81 & 6.81 & 68.78 & 0/3 & 1/3\\
        600 & 20 & 4 & & 30.70 & 10.22 & 66.71 & 0/3 & 0/3\\
        800 & 20 & 4 & & 32.31 & 7.54 & 76.66 & 0/3 & 0/3\\
        \cmidrule{4-9}
        & & & \textbf{Average} & \textbf{28.27} & 
        \textbf{8.19} & \textbf{70.03} & \textbf{0/9}& \textbf{1/9}\\
        \midrule
        400 & 20 & 5 & & 17.36 & 8.05 & 53.63 &0/3&0/3\\
        600 & 20 & 5 & & 25.38 & 12.36 & 51.30 &0/3&0/3\\
        800 & 20 & 5 & & 28.24 & 13.32 & 52.83 &0/3 &0/3\\
        \cmidrule{4-9}
        & & & \textbf{Average} & \textbf{23.66} & \textbf{11.25} & \textbf{52.45} & \textbf{0/9} & \textbf{0/9}\\
        \midrule
        & \multicolumn{3}{c}{} & \textbf{29.07} & \textbf{7.47} & \textbf{74.30} & \textbf{0/27} & \textbf{6/27}\\
        \bottomrule
    \end{tabular}
\end{table}

\Cref{fig:convergence_process} reports the progression of upper and lower bounds for both BD with VIs and the solver on three large instances with $S=3$, where BD closes the optimality gap. 
In \Cref{fig:bd_400,fig:bd_600,fig:bd_800} we observe a steady and rapid progression of both upper (dual) and lower (primal) bounds when using BD. At the same time, \Cref{fig:mip_400,fig:mip_600,fig:mip_800} illustrate that, for the same instances, the solver is able to improve the primal bound but fails to improve the dual bound. Particularly, the solver finds primal solutions of quality comparable to those found by BD. Nevertheless, BD typically obtains such solutions faster. We further observe that the optimality gap of the solver tends to increase with the number of customers. This indicates that part of the explanation of the gap is to be found in the size of the LP relaxation. This further justifies the use of a decomposition method.
 
\captionsetup[subfigure]{font=scriptsize}
\begin{figure}[t]
    \centering
    \begin{subfigure}[b]{0.3\textwidth}
    \includegraphics[width=\textwidth]{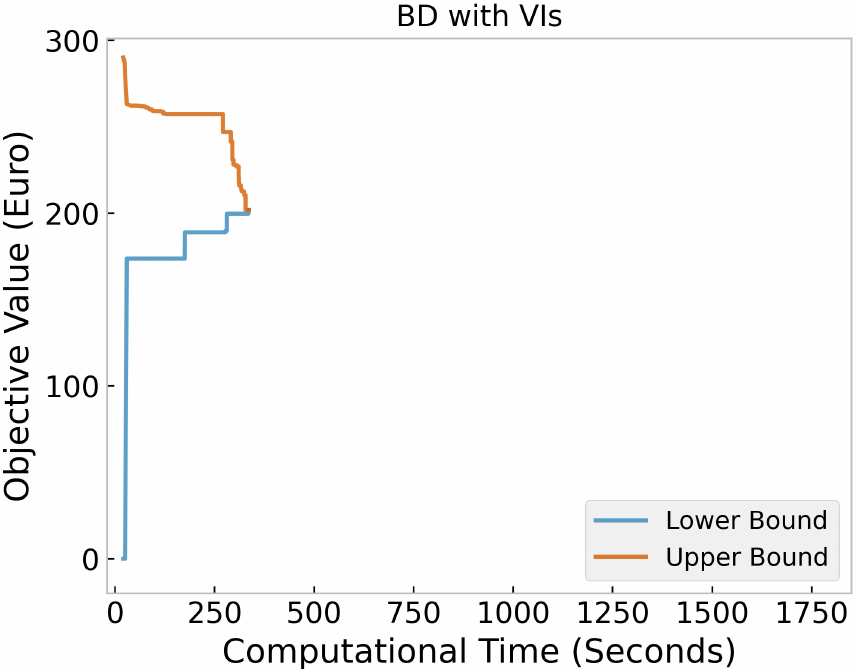}
    \caption{$\vert\set{K}\vert=400$, $\lvert\mathcal{V}\rvert=200$}
    \label{fig:bd_400}
    \end{subfigure}
    \hfill
    \begin{subfigure}[b]{0.3\textwidth}
    \includegraphics[width=\textwidth]{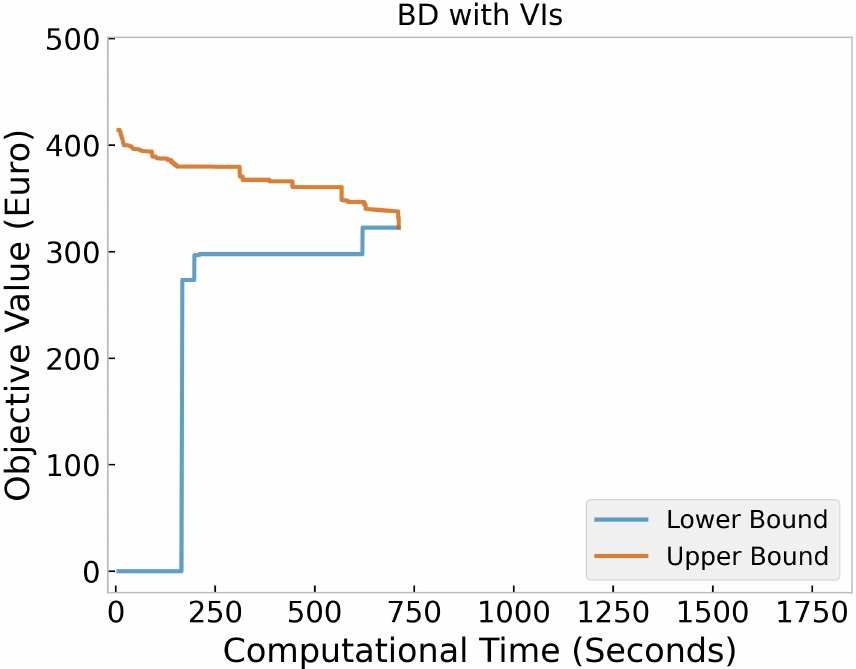}
    \caption{$\vert\set{K}\vert=600$, $\lvert\mathcal{V}\rvert=300$}
    \label{fig:bd_600}
    \end{subfigure}
    \hfill
    \begin{subfigure}[b]{0.3\textwidth}
    \includegraphics[width=\textwidth]{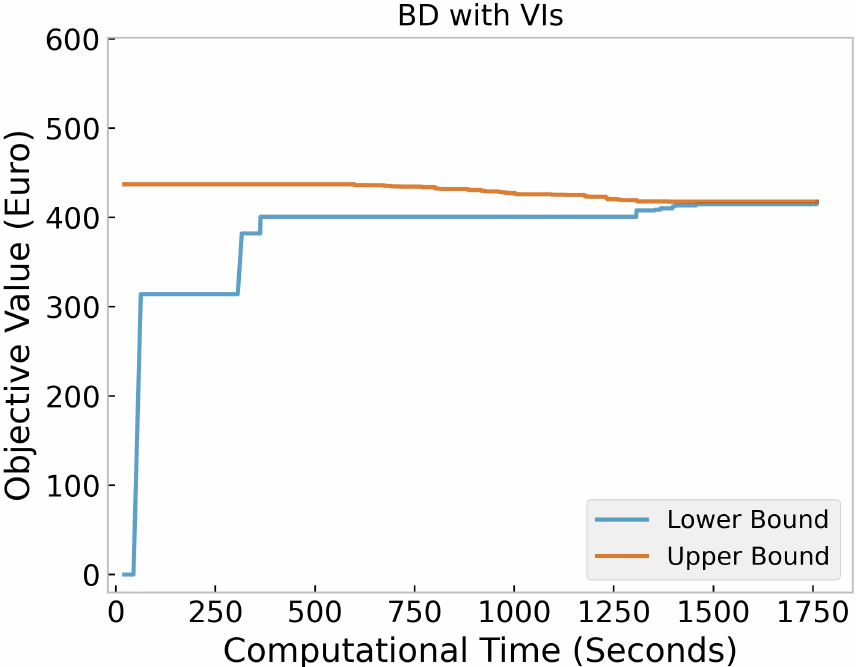}
    \caption{$\vert\set{K}\vert=800$, $\lvert\mathcal{V}\rvert=400$}
    \label{fig:bd_800}
    \end{subfigure}

    \medskip

    \begin{subfigure}[b]{0.3\textwidth}
    \includegraphics[width=\textwidth]{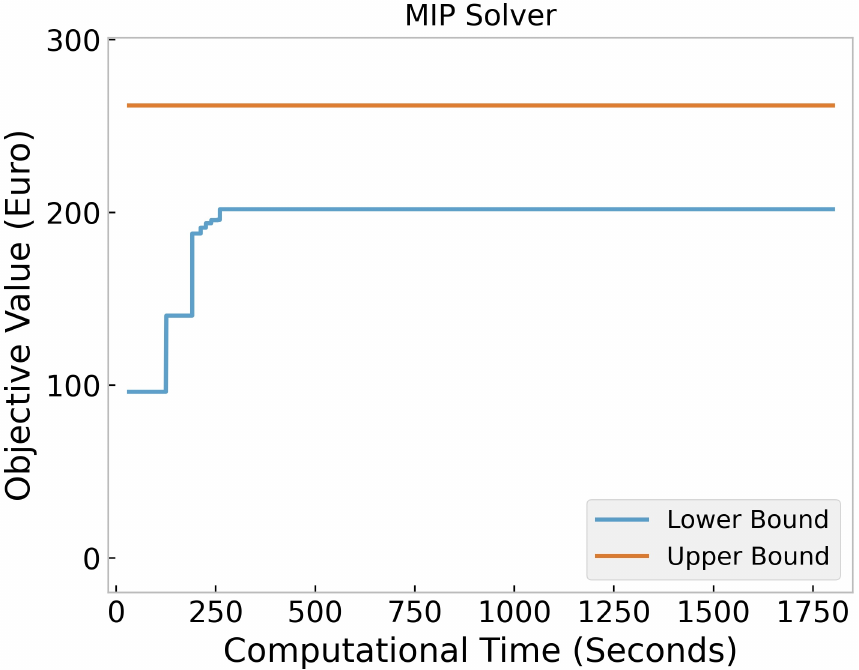}
    \caption{$\vert\set{K}\vert=400$, $\lvert\mathcal{V}\rvert=200$}
    \label{fig:mip_400}
    \end{subfigure}
    \hfill
    \begin{subfigure}[b]{0.3\textwidth}
    \includegraphics[width=\textwidth]{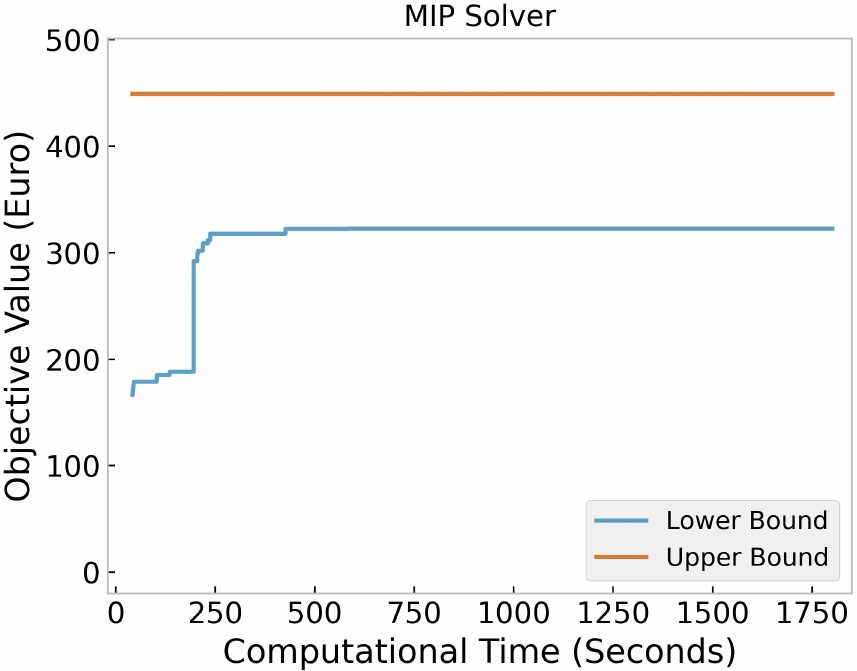}
    \caption{$\vert\set{K}\vert=600$, $\lvert\mathcal{V}\rvert=300$}
    \label{fig:mip_600}
    \end{subfigure}
    \hfill
    \begin{subfigure}[b]{0.3\textwidth}
    \includegraphics[width=\textwidth]{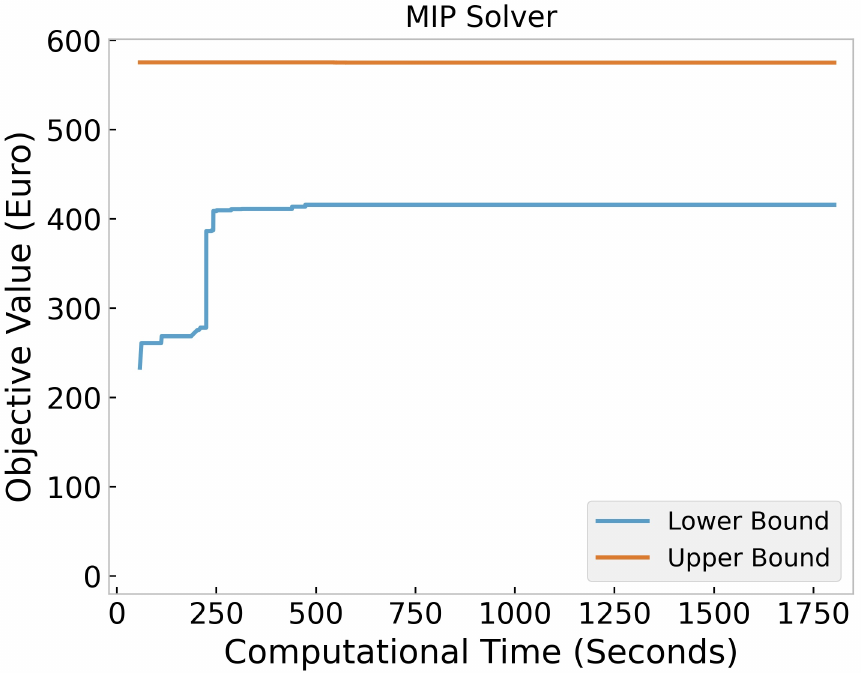}
    \caption{$\vert\set{K}\vert=800$, $\lvert\mathcal{V}\rvert=400$}
    \label{fig:mip_800}
    \end{subfigure}
    \caption{Progression of upper and lower bounds for the large instances with three zones.}
    \label{fig:convergence_process}
\end{figure}

\subsection{Result analysis} \label{sec: result analysis}
In this section, we provide some evidence on how different parameters, such as the number of zones, customers and fleet size, determine business performance.

We begin by analyzing the impact of the number of customers and the number of vehicles on profits. 
\Cref{fig: zone_profit} reports the average profit as a function of the number of customers and fleet size. Particularly, for $\vert\set{K}\vert=400$ we consider three fleet sizes, namely $(100,150,200)$ which we refer to a low, medium and how availability. Similarly, for $\vert\set{K}\vert=600$, the three fleet sizes are $(150,200,300)$ and for $\vert\set{K}\vert=800$ the three fleet sizes are $(200,300,400)$.
\Cref{fig: zone_profit} illustrates that the profit grows linearly with the number of customers, regardless of the number of zones $S$. When the customers volume doubles from $400$ to $800$, the profit increases by $102.85\%$, $102.46\%$, and $93.71\%$ for the $3$-, $4$-, and $5$-zone partitions, respectively. Profits grow also with the fleet size, though the pattern is less regular. In general, a medium level of vehicle availability suffices to secure average profit returns, given a certain customer volume.
Finally, we observe that a finer partition of the carsharing stations (i.e., larger $S$) yields an increase of the average profit. A finer partition allows the service provider to better adapt prices to the specific characteristics of the customers in different parts of the city.

\begin{figure}[t]
    \centering
    \begin{subfigure}[b]{0.325\textwidth}
    \includegraphics[width=\textwidth]{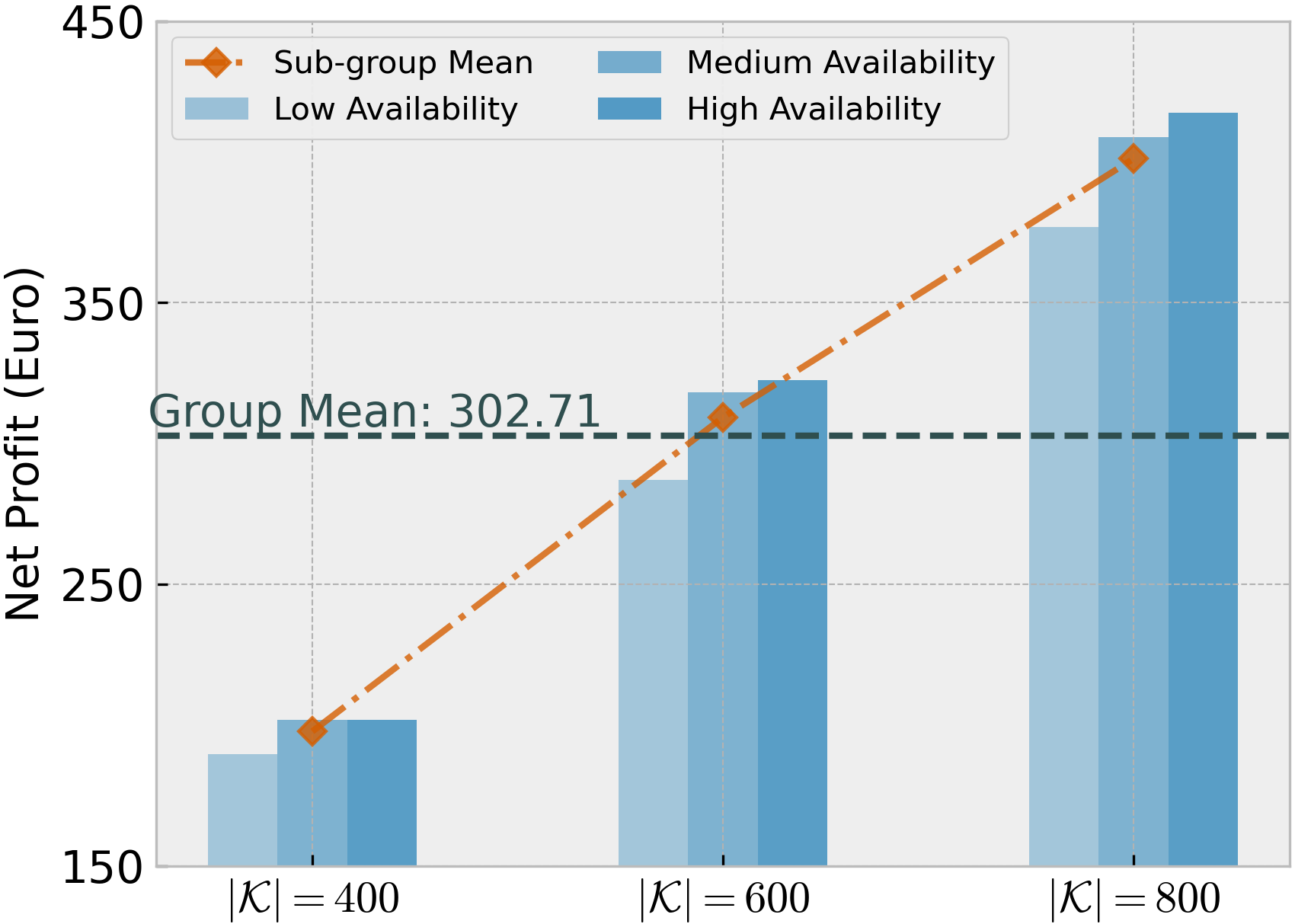}
    \caption{$S=3$}
    \label{fig: 3_zone_profit}
    \end{subfigure}
    \hfill
    \begin{subfigure}[b]{0.325\textwidth}
    \includegraphics[width=\textwidth]{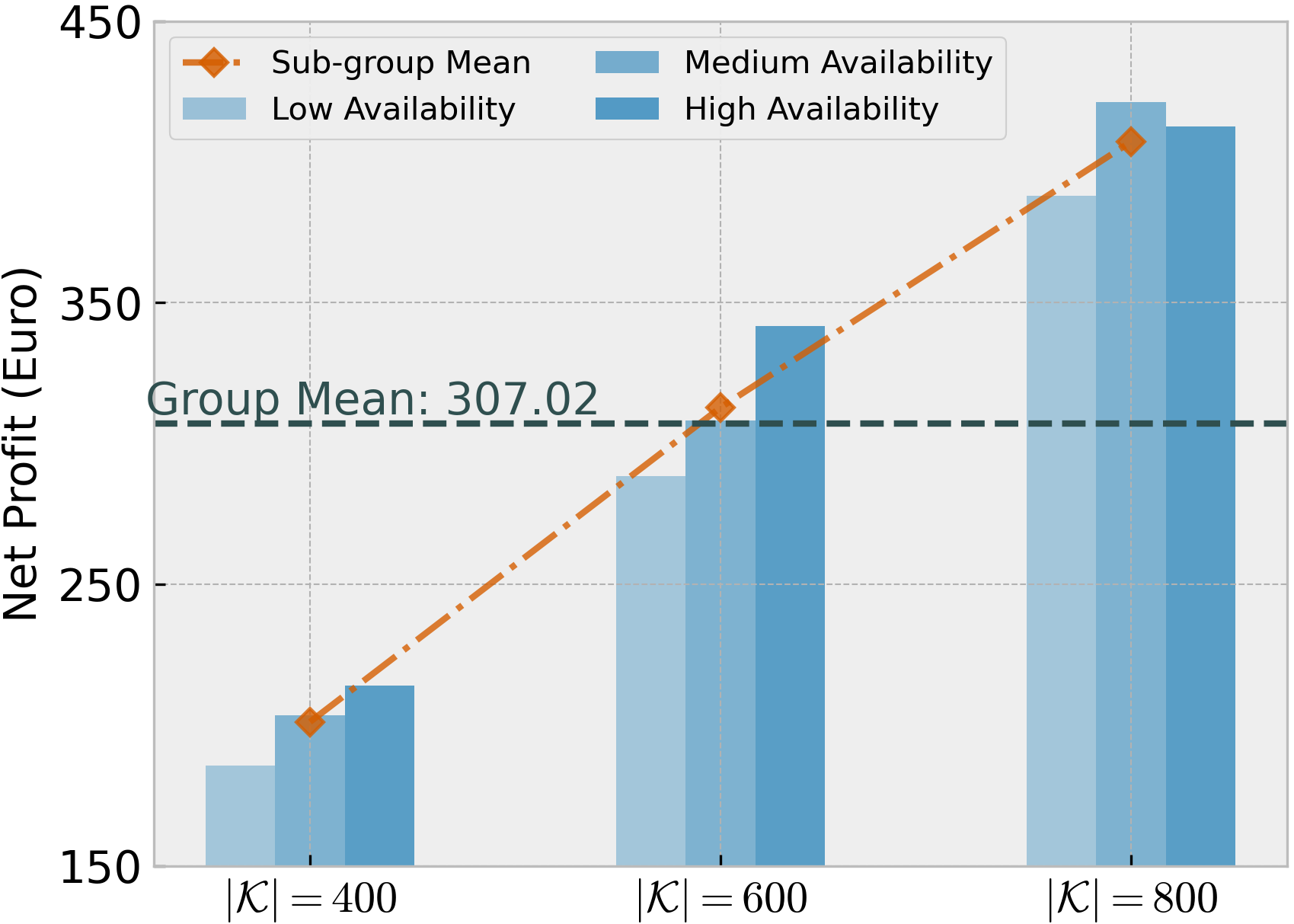}
    \caption{$S=4$}
    \label{fig: 4_zone_profit}
    \end{subfigure}
    \hfill
    \begin{subfigure}[b]{0.325\textwidth}
    \includegraphics[width=\textwidth]{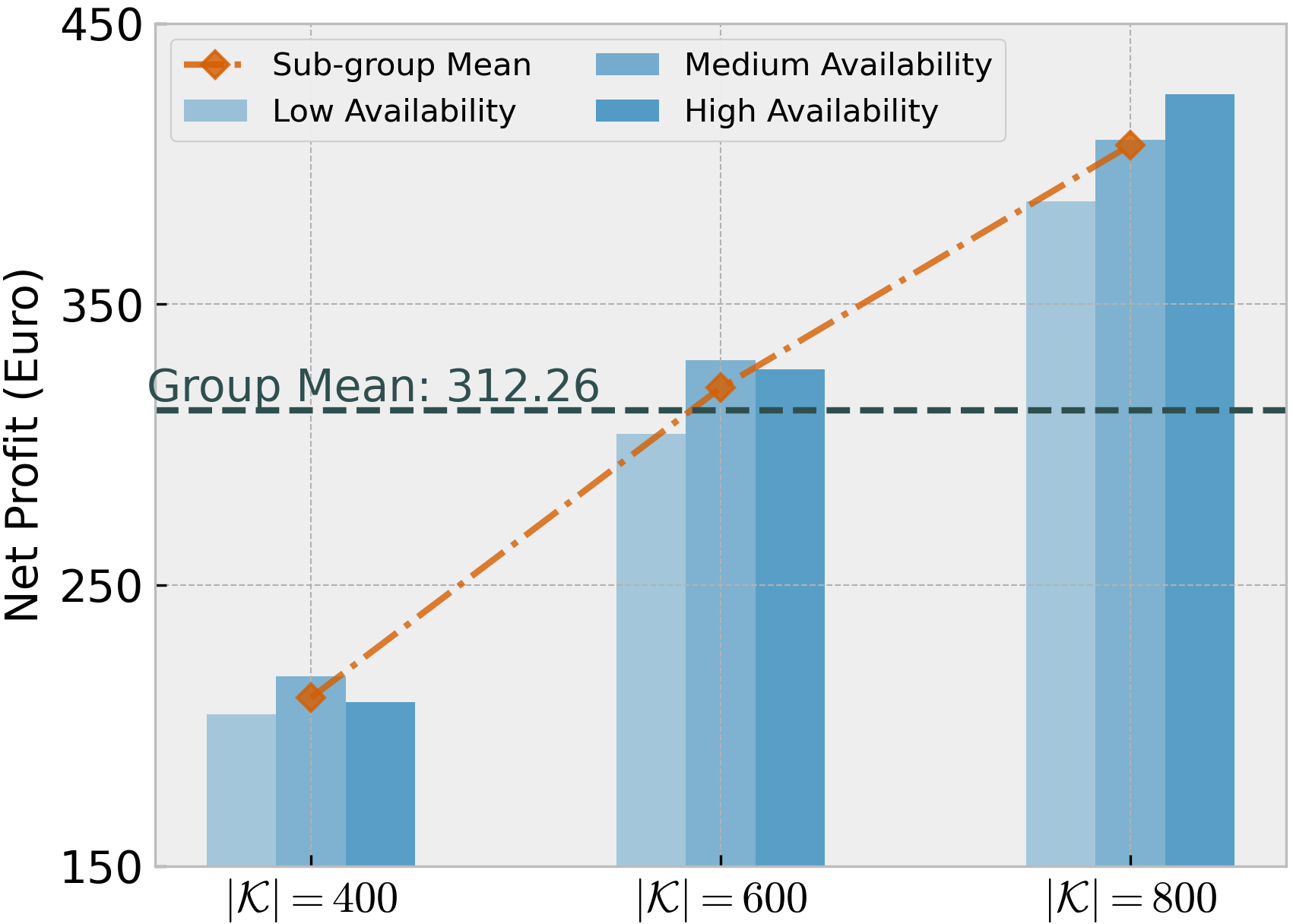}
    \caption{$S=5$}
    \label{fig: 5_zone_profit}
    \end{subfigure}
    \caption{Carsharing system profits under various scales of zone numbers, customer volume, and vehicle availability.}
    \label{fig: zone_profit}
\end{figure}

Following, we evaluate the performance of the carsharing system when using the optimal zonification determined by model \eqref{eq:pricing_problem} (we assume $S=3$ zones) compared two benchmarks. The first benchmark is obtained by partitioning the carsharing stations according to the zip codes and keeping that partition fixed when optimally setting prices. Also in this case we use $S=3$ zones.  We refer to this benchmark as the \textit{zip-code partition} benchmark. The second benchmark consists of not partitioning the carsharing stations and thus optimally choosing the drop-off fee that must apply between all pairs of carsharing stations. This benchmark is referred to as the \textit{no-partition} benchmark. 
\Cref{tab:css performance} reports profits and service rates with $400$, $600$ and $800$ customers assuming a fleet size of $200$, $300$, and $400$, respectively.
Service rates are calculated as the number of customers served over the total number of requests $\vert\set{R}\vert$.

The results illustrate that by implementing an optimal partition, and thus ensuring co-optimized zones and prices, the profit increases substantially both with respect to the zip-code partition and to the no-partition benchmark. Specifically, when compared to the zip-code partition benchmark, the operating profit with the optimal partition  increases by $10.61\%$, $6.73\%$, and $5.54\%$ for cases with $400$, $600$, and $800$ customers, respectively. The improvement becomes even more pronounced when compared with the no-partition scenario. Particularly, the profit increases by $32.48\%$, $26.51\%$, and $21.88\%$ for three different customer volumes.

When using the optimal partition, the average service rate is $76.43\%$, slightly below the service rate of the zip-code partition benchmark, and significantly higher than the no-partition benchmark.
We recall that model \eqref{eq:pricing_problem} is designed to maximize profits, thus service rates are not directly targeted. Nevertheless, the model and solution algorithm can be easily adapted to include measures of performance based on service rates. This would only require changing the specification of $Q(a,\lambda,\alpha)$.

\begin{table}[t]
\scriptsize
\caption{Service rates and profits for the optimal partition compared to the zip-code partition and no-partition benchmarks. The results are obtained with $S=3$ zones.}
\centering
    \begin{tabular}{ccccccc}
    \toprule
         \multirow{2}{*}{$\vert\mathcal{K}\vert$}& \multicolumn{2}{c}{Optimal partition} & \multicolumn{2}{c}{Zip-code partition} & \multicolumn{2}{c}{No-partition}\\
         \cmidrule{2-7}
         &  Service rate (\%) & Profit (Euro) & Service rate (\%) & Profit (Euro) & Service rate (\%) & Profit (Euro) \\
    \midrule
         $400$& 78.13 & 201.82 & 81.25 & 182.46 & 56.25 &152.34 \\
         $600$ & 78.63 & 322.58 & 79.49 & 302.24 & 49.57 & 254.98 \\
         $800$ & 72.53 & 417.56 & 81.69 & 395.64 & 55.63 & 342.60\\
    \bottomrule
    \end{tabular}
    \label{tab:css performance}
\end{table}

Part of the explanation beyond the increase in profit when using the optimal partition compared to the zip-code partition can be found in \Cref{fig: service rate}.
The figure reports the percentage of customers served at different price levels. It emerges that the number of customers served at a negative drop-off fee (Euro $-1$) is much higher when using the zip-code partition compared to the optimal partition. This also contributes to explaining the slightly higher service rates generated by the zip-code partition. When using the optimal partition, significantly fewer customers are served at the negative drop-off fee. The customers are instead more evenly spread across the different drop-off fees.  This is mainly due to the ability to adjust the partition of the business area, which allows determining zones that better capture the price preferences of the customers.
\begin{figure}[ht]
    \centering
    \begin{subfigure}[b]{0.325\textwidth}
    \includegraphics[width=\textwidth]{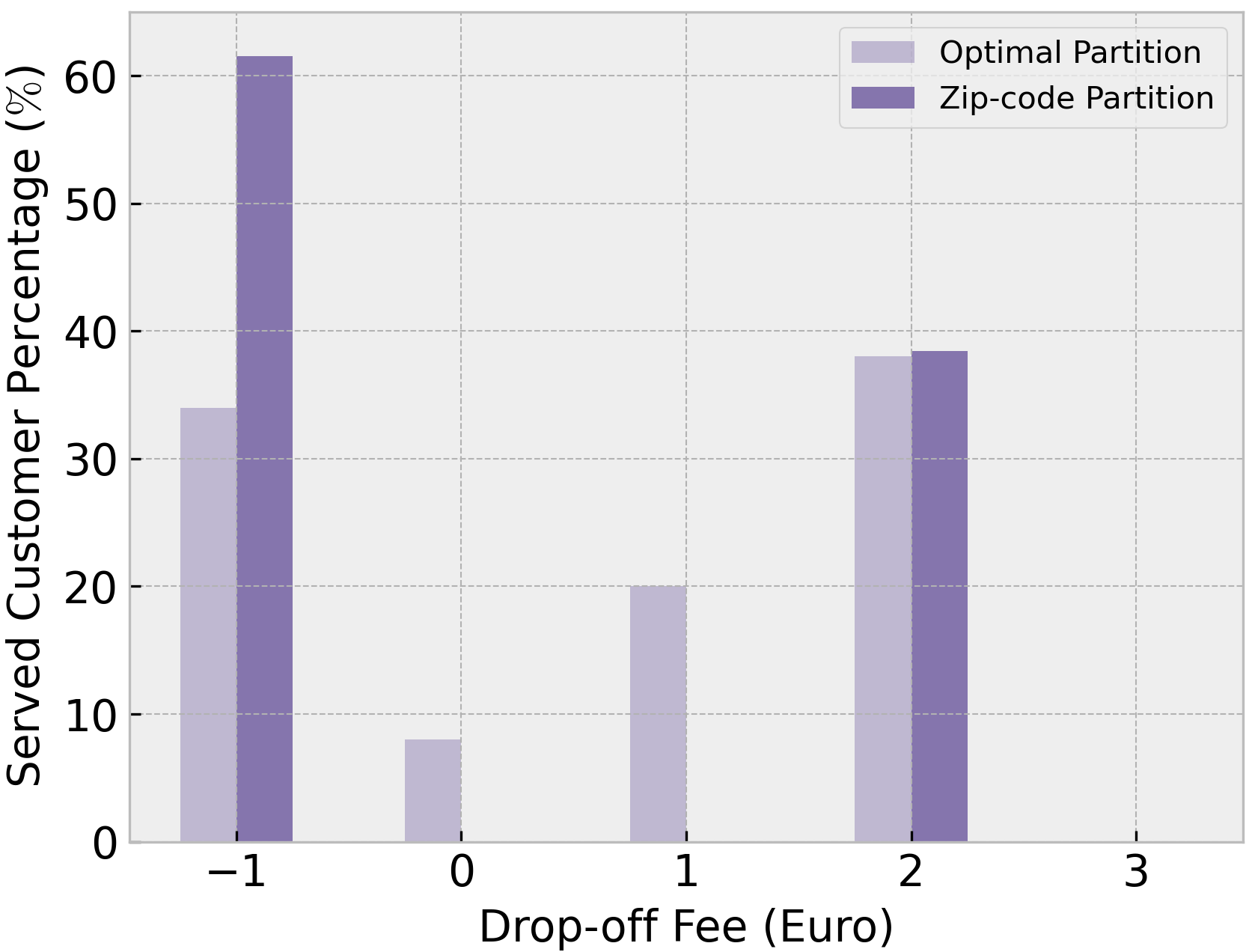}
    \caption{$\vert\mathcal{K}\vert=400$}
    \label{fig: service_rate_400}
    \end{subfigure}
    \hfill
    \begin{subfigure}[b]{0.325\textwidth}
    \includegraphics[width=\textwidth]{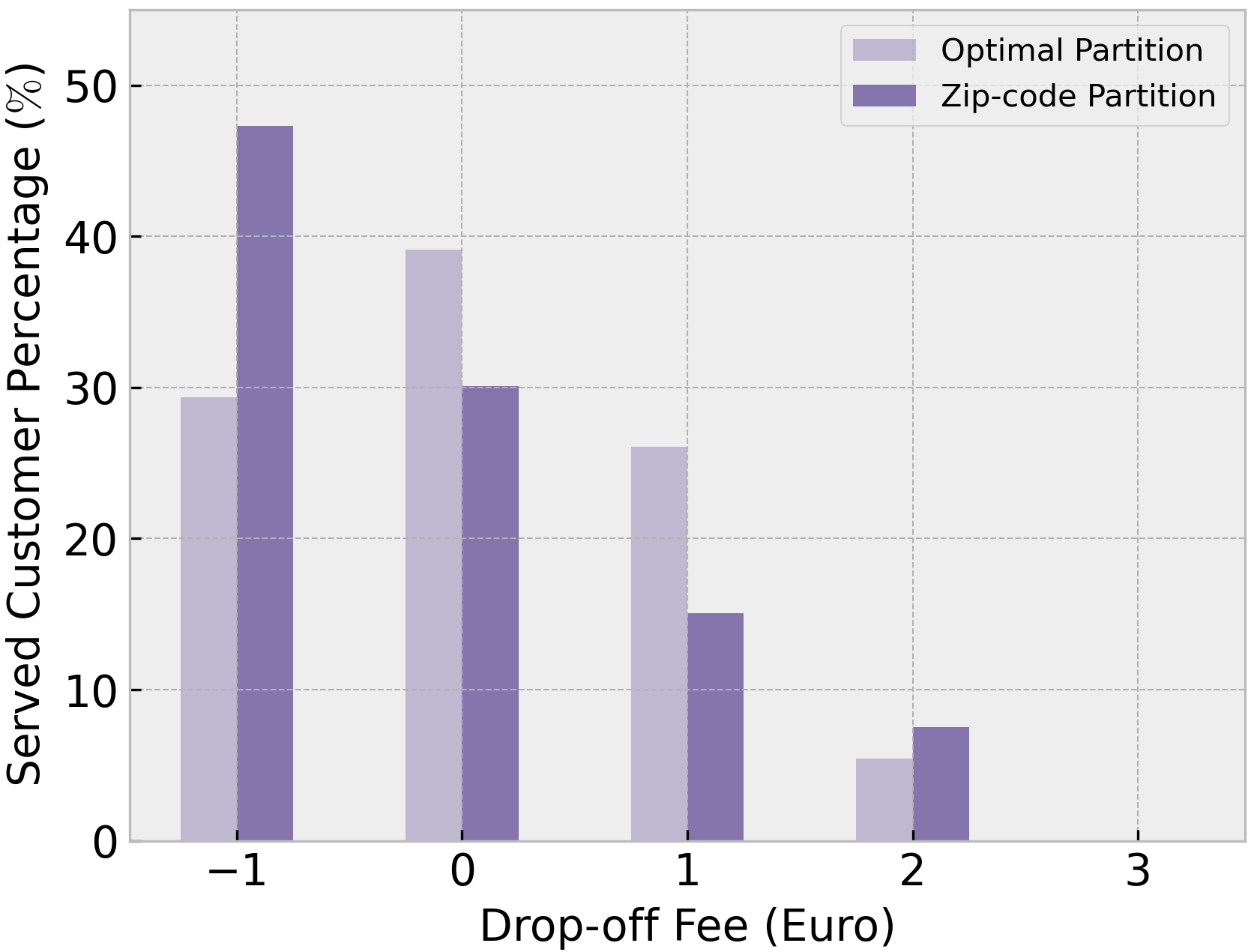}
    \caption{$\vert\mathcal{K}\vert=600$}
    \label{fig: service_rate_600}
    \end{subfigure}
    \hfill
    \begin{subfigure}[b]{0.325\textwidth}
    \includegraphics[width=\textwidth]{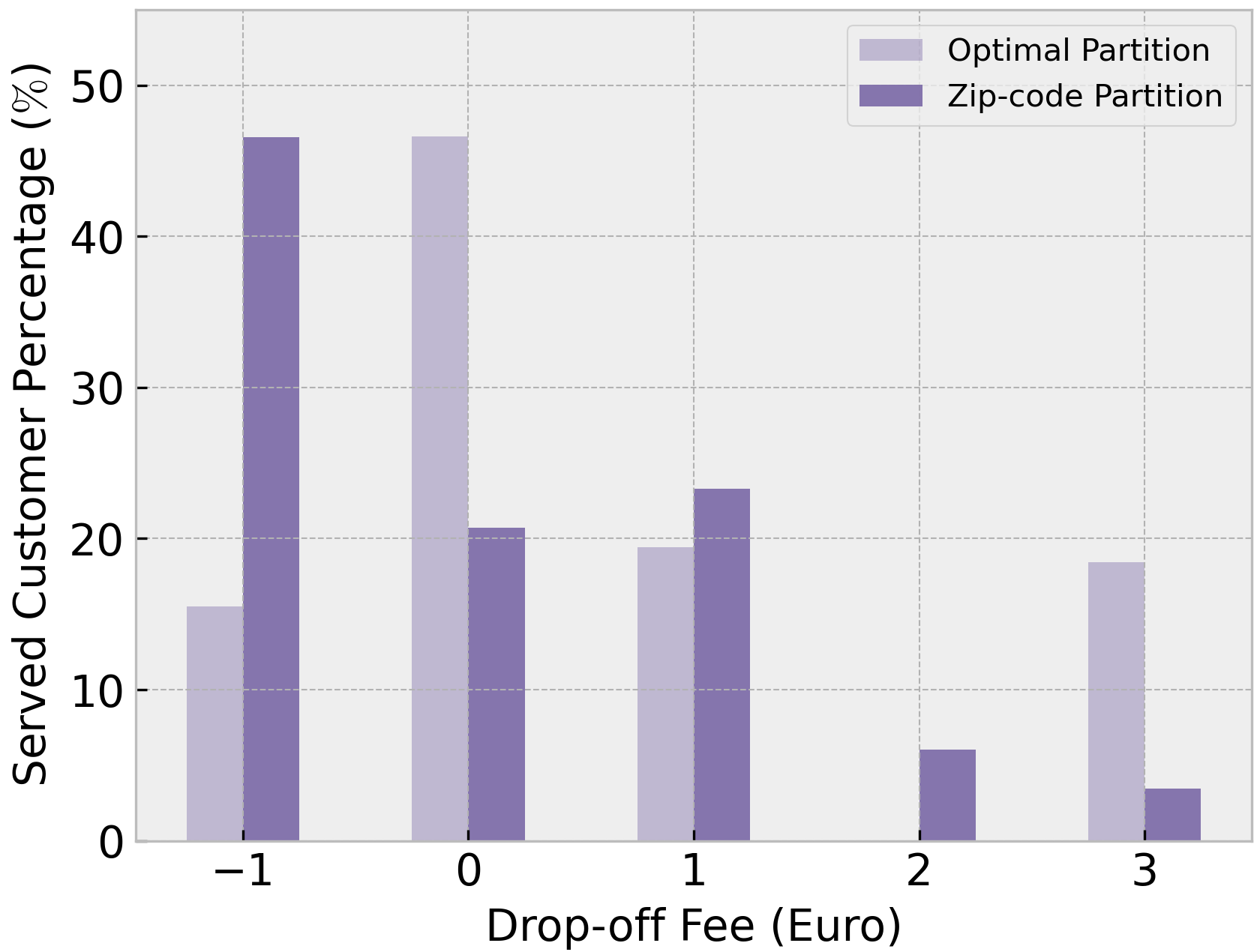}
    \caption{$\vert\mathcal{K}\vert=800$}
    \label{fig: service_rate_800}
    \end{subfigure}
    \caption{Percentage of the customers served at the different available price levels.}
    \label{fig: service rate}
\end{figure}

\color{black}
\section{Conclusions} \label{sec: conclusion}
In this paper, we studied the problem of jointly deciding price zones and trip prices in a carsharing service. To partition the carsharing stations into distinct pricing zones we introduced a special type of tessellation and proved that such tessellation fulfills the requirement that pricing zones are ``visually disjoint''. The problem of choosing such a tessellation, and optimally assigning prices has been formulated as a binary integer programming problem. To solve the problem we designed a tailored integer Benders decomposition, which incorporates a number of problem-specific improvements.

We performed extensive experiments on instances based on a carsharing system operated in Copenhagen. The results illustrate that the proposed Benders decomposition significantly outperforms a state-of-the-art solver on instances of size comparable with business practices. Particularly, our method reports an average $7.47\%$ optimality gap, which is significantly lower to that reported by the solver ($29.07\%$). Furthermore, our method closes the optimality gap on a number of instances while no instance is solved to optimality by the solver.  

Furthermore, our results illustrate that the optimal zonification yields profits that are $7.01\%$ higher (on average) compared to the zonification determined by zip codes and $25.61\%$ higher compared to the prevailing practice of having only one pricing zone. Notably, the optimal zonification yields also high service rates, by serving $76.43\%$ of the customers, on average.

\newpage
\begin{APPENDICES}
\crefalias{section}{appendix}
\section{Algorithm for the solution {$Q(a,\lambda,\alpha)$}}\label{app:alg}
We present here \Cref{greedy_alg} to solve problem \eqref{eq:allocation} to optimality.
\begin{algorithm}[h]
    \footnotesize
	\caption{Exact computation of $Q(a, \lambda, \alpha)$}
	\begin{algorithmic}[1]
	\State \textbf{Input:} $a$, $\lambda$, $\alpha$
	\State $\mathcal{V}^A \gets \mathcal{V}$ \Comment{$\mathcal{V}^A$ is the set of remaining available vehicles.}
	\State $y_{vrl} \gets 0, \forall v \in \mathcal{V}, r \in \mathcal{R}, l \in \mathcal{L}_r$
	\State $Q(a, \lambda, \alpha) \gets 0$
	\State Sort the requests $\mathcal{R}$ in non-decreasing order of customer index $k(r)$
	\For{request $r \in \mathcal{R}$}
	\State $L_{i(r),j(r)} = \sum_{l \in \mathcal{L}} L(l) \alpha_{i(r),j(r),l}$ \Comment{The fee applied between $i(r)$ and $j(r)$.}
	\If{$L_{i(r),j(r)} \leq l(r)$}
	\For {$v \in \mathcal{V}^A$}
	\If {$G_{v,i(r)} = 1$}
    \State $\set{V}^A\gets \set{V}^A\setminus\{v\}$ \Comment{Vehicle $v$ becomes unavailable.}
	\State $y_{v,r,L_{i(r),j(r)}} \gets 1 $
	\State $Q(a, \lambda, \alpha) \gets Q(a, \lambda, \alpha) + R^N_{r,L_{i(r),j(r)}}$ \Comment{$R^N_{r,L_{i(r),j(r)}}$ is the net revenue of $r$ at fee $L_{i(r),j(r)}$.}
	\EndIf
	\EndFor
	\EndIf
	\EndFor
	\State \Return $Q(a, \lambda, \alpha)$, and $y_{vrl} ~~\forall v \in \mathcal{V}, r \in \mathcal{R}, l \in \mathcal{L}_r$
	\end{algorithmic} \label{greedy_alg}
\end{algorithm}
\Cref{greedy_alg} first initializes the set of available shared vehicles $\mathcal{V}^A$, the solution $y_{vrl}$ and the objective value $Q(a, \lambda, \alpha)$. All requests in the set $\mathcal{R}$ are sorted in a non-decreasing order of customer index $k(r)$. Following the rule of ``first-come, first-served", the algorithm iteratively checks whether the fee applied between the origin and destination is accepted and whether there is a vehicle available in the origin zone. If the fee level $L_{i(r),j(r)}$ between request $r$'s origin station $i(r)$ and destination station $j(r)$ is lower than the highest acceptable pricing level $l(r)$ and there is at least one available vehicle $v$ at customer's origin $i(r)$ (i.e., a vehicle with $G_{v,i(r)}=1$), customer $r$ will be served at pricing level $L_{i(r),j(r)}$. We then remove the used vehicle $v$ from the available vehicles, set the value of $y_{v,r,L_{i(r),j(r)}}$ to $1$, and update $Q(a, \lambda, \alpha)$.

\section{Extensive MILP formulation}\label{app: MILP form}
In this appendix, we provide the extensive MILP formulation of problem \eqref{eq:pricing_problem} when considering the specific $Q(a, \lambda, \alpha)$ function defined in \Cref{sec:customer_choices} as follows
\begin{subequations}
\begin{align}
\max_{a, \lambda, \alpha} &~~ \sum_{r \in \mathcal{R}} \sum_{v \in \mathcal{V}} \sum_{l \in \mathcal{L}_r} R_{rl}^N y_{vrl}\\
& \sum_{j\in\mathcal{I}}a_{ii} = S \\
& \sum_{j\in\mathcal{I}}a_{ij} =1 &\forall i\in\mathcal{I}\\
& a_{ij}\leq a_{jj} &\forall i,j\in\mathcal{I}\\
& d(i,j_1)a_{i,j_1} \leq d(i,j_2)a_{j_2,j_2}+ d(i,j_1) (1-a_{j_2,j_2}) &\forall i,j_1,j_2\in\mathcal{I}\\
& \sum_{l \in \mathcal{L}} \lambda_{ijl} \geq a_{ii} + a_{jj} - 1 &\forall i, j \in \mathcal{I}\\
& \sum_{l \in \mathcal{L}} \lambda_{ijl} \leq a_{ii} &\forall  i, j \in \mathcal{I}\\
& \sum_{l \in \mathcal{L}} \lambda_{ijl} \leq a_{jj} &\forall  i, j \in \mathcal{I}\\
& a_{i_1,j_1} + a_{i_2,j_2} + \lambda_{j_1,j_2,l} \leq \alpha_{i_1,i_2,l} + 2 &\forall i_1, i_2, j_1, j_2 \in \mathcal{I}, \forall l \in \mathcal{L}\\
& \sum_{l \in \mathcal{L}} \alpha_{ijl} = 1 &\forall i,j \in \mathcal{I}\\
&\sum_{v\in \mathcal{V}} \sum_{l \in \mathcal{L}_r} y_{vrl} \leq 1 &\forall r \in \mathcal{R}\\
&\sum_{r \in \mathcal{R}} \sum_{l \in \mathcal{L}_r} y_{vrl} \leq 1 &\forall v \in \mathcal{V}\\
& \sum_{v \in \mathcal{V}} y_{vrl} \leq \alpha_{i(r),j(r),l} &\forall r \in \mathcal{R}, l \in \mathcal{L}_r\\
&\sum_{l \in \mathcal{L}_{r}}y_{vrl} + \sum_{r_1 \in \mathcal{R}_r} \sum_{l_1 \in \mathcal{L}_{r_1}} y_{v,{r_1},l} \leq G_{v,i({r})} &\forall r \in \mathcal{R}, v \in \mathcal{V}\\
& y_{vrl} + \sum_{r_1 \in \mathcal{R}_{r}} \sum_{l_1 \in \mathcal{L}_{r_1}} y_{v,r_1,l_1} + \sum_{v_1 \in \mathcal{V}\setminus\{v\}} y_{{v_1},r,l} \geq \alpha_{i(r),j(r),l} + G_{v,i(r)} - 1 &\forall r \in \mathcal{R}, v \in \mathcal{V}, l \in \mathcal{L}_{r}\\
&a_{ij}\in\{0,1\} & \forall i,j\in\set{I}\\
& \lambda_{ijl}, \alpha_{ijl} \in \{0,1\} &\forall i,j \in \mathcal{I}, \forall l \in \mathcal{L}\\
& y_{vrl} \in \{0,1\} &\forall r \in \mathcal{R}, v \in \mathcal{V}, l \in \mathcal{L}_{r}
\end{align} \label{compact_formulation}
\end{subequations}

\end{APPENDICES}

\ACKNOWLEDGMENT{This research is supported by the research project ``Shared mobility: Towards sustainable urban transport'' (grant no. 1127-00176B) funded by Danmarks Frie Forskningsfond (DFF).}

\bibliography{reference}
\end{document}